\documentclass[a4paper,12pt, psamsfonts, reqno]{amsart}
\usepackage{verbatim}

\usepackage[switch,pagewise,running]{lineno}
\usepackage[title]{appendix}
\usepackage{amsmath}
\usepackage{amssymb}
\usepackage{textcomp}
\usepackage{amsthm}
\usepackage{amsfonts}
\usepackage[all]{xy}
\usepackage{mathrsfs}
 \usepackage{graphicx}
 \usepackage{epstopdf}
 \usepackage{float}
 \usepackage{graphicx}
 
 \usepackage{epstopdf}
 \usepackage[T1]{fontenc}
 \usepackage{hyperref}
 
\newtheorem{theo}{Th\'eor\'eme}[section]

\def\A{\mathbb{A}}

\def\C{\mathbb{C}}

\def\D{\mathscr{D}}

\def\N{\mathbb{N}}

\def\Q{\mathbb{Q}}

\def\aR{\mathcal{R}}
\def\K{\mathbf{K}}

\def\Z{\mathbb{Z}}

\def\a{\mathbf{a}}
\def\l{\ell}
\def\e{\mathbf{e}}
\def\c{\mathbf{c}}

\def\f{\mathbf{f}}
\def\b{\mathbf{b}}
\def\d{\mathbf{d}}

\def\O{\mathcal{O}}
\def\line{\overline}

\def\Id{\mathop{\mathrm{Id}}\nolimits}
\def\Hom{\mathop{\mathrm{Hom}}\nolimits}

\def\ker{\mathop{\mathrm{ker}}\nolimits}

\def\dim{\mathop{\mathrm{dim}}\nolimits}

\def\wt{\mathop{\mathrm{wt}}\nolimits}

\def\GL{\mathop{\mathrm{GL}}\nolimits}
\def\End{\mathop{\mathrm{End}}\nolimits}
\def\hat{\widehat}

\def\remk{\noindent\textit{Remark:~}}

\newtheorem{prop}[theo]{Proposition}
\newtheorem{prop-def}[theo]{Proposition-Definition}
\newtheorem{def-prop}[theo]{Definition-Proposition}
\newtheorem{cor}[theo]{Corollary}

\newtheorem{lemma}[theo]{Lemma}
\newtheorem{teo}[theo]{Theorem}
\newtheorem{definition}[theo]{Definition}

\newtheorem{notation}[theo]{Notation}
\newtheorem{example}[theo]{Example}

\allowdisplaybreaks

\everymath{\displaystyle}

\input xy
\xyoption{all}

\title[Derivatives]{A geometric study of BZ operator on representations of $\GL_n$ over non-archimedean field}
\author{Taiwang DENG}
\address{ 
Beijing Institute of Mathematical Sciences and Applications (BIMSA), Traffic Light Gate, No. 544 Hefangkou Village, Huairou District, Beijing}
\email{dengtaiw@bimsa.cn}
\date{}


\begin{document}

\begin{abstract}
In this article, we geometrically study the partial Bernstein-Zelevinsky operator introduced in the author's thesis, which generalizes the original Bernstein-Zelevinsky operator. We relate the partial Bernstein-Zelevinsky operator to the geometric induction of Lusztig and then perform explicit computations in special cases. Finally, we develop a symmetric reduction to the previously mentioned special cases
\end{abstract}

\maketitle

\tableofcontents
\section{Introduction}
Zelevinsky \cite{Z2} classifies
the admissible irreducible representations of $\GL_n(F)$ 
in terms of multi-segments. More precisely, 
given a multi-segment $\a$, one can attach
to it an irreducible representation $L_{\a}$, 
described as the unique irreducible sub-representation in 
some standard representation $\pi(\a)$ constructed by
parabolic induction. In other words, in the
Grothendieck group $\aR$ of the category of admissible 
representations, we have
\[
\pi(\a)=L_{\a}+\sum_{\b<\a} m(\b, \a)L_{\b}, \quad m(\b, \a)\in \Z_{\geq 0}, 
\]
where "$<$" is suitable partial order imposed on the set of multi-segments.
In \cite[4.5]{Z1}, Bernstein and Zelevinsky define an operator  $\mathscr{D}$ (We call it BZ  operator) to be 
 an algebra  homomorphism 
 \[
  \mathscr{D}: \mathcal{R}\rightarrow \mathcal{R},
 \]
which plays a crucial role in Zelevinsky's classification theorem. We introduced a partial analogue of the BZ operator
in \cite[Definition 2.19]{Deng23} (this is part of the author's thesis \cite{Deng16}).

In this article, 
we study the problem of computing the partial BZ operator
$\D^k(L_{\a})$ of the irreducible representation 
$L_{\a}$ attached to a multisegment $\a$. There are two motivations for this. The first  is to use these computations
to calculate the multiplicities in certain induced representation  $L_{\a} \times L_{\b}$ in a forthcoming work. The second
is also related to  forthcoming work, in which we study the relations between the category of equivariant perverse sheaves on quiver 
varieties (in the sense of \cite[\S 4]{Deng23}) and the category of perverse sheaves on flag varieties base on ideas of \cite[\S 4]{Deng23}.
We understand that this is a dual version of \cite{Chan23}. We further want to relate the partial BZ operator to the translation functor
in representation of real Lie groups based on work of the present paper.

Let us discuss in more details about the content of the paper.
From Zelevinsky's classification theorem, we can write 
\[
 L_{\a}=\sum_{\b}\tilde{m}_{\b, \a}\pi(\a).
\]
Thus, the task boils down to calculating:
\[
 \D^k(\pi(\a))=\sum_{\b} n_{\b, \a} L_{\b},~ \qquad n_{\b, \a}\geq 0.
\]

We first  introduce a new poset structure $\preceq_k$
on the set of multisegments so that  we show the equivalence 
between $n_{\b, \a}>0$ and $\b\preceq_k \a$, cf. Proposition \ref{prop: 7.4.2}. 

The primary result of this paper 
is the interpretation of the coefficient 
$n_{\b, \a}$ as the value at $q=1$ of some 
Poincar\'e series of the Lusztig product of 
two explicit perverse sheaves on the space of graded nilpotent classes, cf. Theorem
\ref{prop: 7.3.8}. This relies crucially on the work of Lusztig, as discussed in section \ref{sec-recall-Lusztig}.
Note that the category of equivariant perverse sheaves should be 
understood as the dual of the category of representations of $\GL_n(F)$. Given that the BZ operator is related to the Jacquet functor, on the dual side, it should relate to geometric induction.

In section \ref{sec-Lusztig-prod}, we compute these Lusztig products as the push forward by a projection 
$\beta''$, cf. Theorem \ref{cor: 7.4.19}, of some concrete perverse sheaf on 
some space of graded nilpotent classes. This is our second main results.

Section \ref{sec-geom-grass}-\ref{section-Grass}  are dedicated to a concrete study of the coefficient $n(\b, \a)$ in the Grassmanian case, from which we deduce some combinatorial formulae for it, cf. Proposition \ref{prop-combinatorial-formula-BZ}.

In section \ref{sec-Para-case},  we continue to study the  case where the 
multisegments are of Grassmanian type. In this case the projection $\beta''$ is simple
(cf. Proposition \ref{prop: 7.7.8}), being the natural projection 
\[
 GL_n/P\rightarrow GL_n/P'
\]
with $P\subseteq P'$ two parabolic subgroups.
The constructions and proofs in this case are very close to the case of 
Grassmanian type.

Finally in the last section \ref{Explict-calculation-irred}, we obtain a complete formula for 
$\D^k(L_{\a})$ in the general case, cf. Corollary \ref{coro-fornula-derivative}. It's worth noting that to achieve this, we require a reduction process to the parabolic cases (cf. Proposition \ref{prop-red-to-para}) in the style of \cite[\S 6.2]{Deng23}. This elucidates the extensive discussion thereof.

\par \vskip 1pc
{\bf Acknowledgements}
This paper is part of my thesis at University Paris 13, which is funded by the program DIM of the
region Ile de France. I would like to thank my advisor Pascal Boyer for his keen interest in this work and 
his continuing support and countless advice.
It was rewritten during my stay as a postdoc at Max Planck Institute for Mathematics and Yau Mathematical Sciences Center, I thank their hospitality. In addition, 
I would like to thank Alberto M\'inguez ,Vincent S\'echerre, Yichao Tian and Bin Xu for their helpful discussions on the subject.


\section{New Poset Structure on Multisegments}

We keep the same notations as in \cite[\S 2]{Deng23}.
Recall that in \cite[Definition 2.12]{Deng23} we have introduced a partial order $\leq$ on the set of multisegments and the finite set
\[
S(\a)=\{\b: \b\leq \a\}
\]
following \cite[7.1]{Z2}. 
We also recall that in \cite[Notation 2.15]{Deng23} the standard representation $\pi(\a)$ was attached to a multisegment
$\a$ and its decomposition in the Grothendieck group into irreducible ones was considered:
\begin{equation}\label{eqn-decom-standard-irr}
  \pi(\a)=\sum_{\b\in S(\a)} m(\b,\a)L_{\b}, \quad m(\b, \a)\geq 0
\end{equation}
where $L_{\a}$ denotes the irreducible representation attached to $\a$.
Here we have taken the cuspidal representation $\rho=1$ to be the trivial representation of $GL_1$, an assumption that we will
make throughout this paper.
In this section we define a new poset structure $\preceq_k$ depending on an 
integer $k$ on
the set of multisegments and show that the term $L_{\b}$ appears in
$\D^k(\pi(\a))$ (cf. \cite[Definition 2.19]{Deng23}) if and only if $\b\preceq_k \a $.

Let $b(\Delta)$ (resp. $e(\Delta)$ )denote the beginning (resp. the end) of a segment $\Delta$.
We also use the notation 
\begin{equation}\label{eqn-set-beginning}
b(\a)( \text{ \ resp. \ } e(\a))
\end{equation}
to denote the multi-set of beginnings (resp. ends) in a multisegment $\a$.
If $\Delta=[i, j]$ we set $\Delta^{-}=[i, j-1]$.
\begin{definition}\label{def: 2.2.3}
Let $k\in \Z$ and $\Delta$ be a segment, we define 
\begin{displaymath}
 \Delta^{(k)}=\left\{ \begin{array}{cc}
         &\Delta^-, \text{ if } e(\Delta)=k;\\
         &\hspace{-0.5cm}\Delta, \text{ otherwise }.
         \end{array}\right.
\end{displaymath}
For a multisegment 
 $
 \a=\{\Delta_{1}, \cdots, \Delta_{r}\},
 $
we define
  \[
  \a^{(k)}=\{\Delta_{1}^{(k)}, \cdots,\Delta_{r}^{(k)} \}.
 \]
\end{definition}

\begin{definition}\label{def-new-partial-order}
For a well ordered multisegment $\a=\{\Delta_1, \cdots, \Delta_s\}$
with $\Delta_1\preceq \cdots \preceq \Delta_s$,
let 
\[
 \a(k):=\{\Delta\in \a: e(\Delta)=k\}=\{\Delta_{i_{0}}, \Delta_{i_0+1}, \cdots, \Delta_{i_1}\}. 
\]
Now let $\Gamma \subseteq \a(k)$, let 
\[
 \a(k)_\Gamma: =(\a(k)\setminus \Gamma)\cup \{\Delta^{(k)}: \Delta\in \Gamma\},
\]
and 
\[
 \a_\Gamma: =(\a\setminus \a(k))\cup \a(k)_\Gamma.
\] 

We say $\b\preceq_k \a$ if there exist a multisegment $\c\in S(\a)$ such that 
\[
 \b\leq \a_\Gamma (\text{ cf. \cite[Definition 2.12]{Deng23}})
\]
for some $\Gamma$.  
\end{definition}

\begin{lemma}
We have 
\addtocounter{theo}{1}
\begin{equation}\label{eq: 7.41}
 \D^k(\pi(\a))=\pi(\a)+\sum_{\Gamma\subseteq \a(k), \Gamma\neq \emptyset}\pi(\a_\Gamma).
\end{equation}  
\end{lemma}

\begin{proof}
 
Let 
\[
\a=\{\Delta_1, \cdots, \Delta_r, \Delta_{r+1}, \cdots, \}.     
\]
Then 
\[
 \pi(\a)=\prod_{i=1}^{r}L_{\Delta_i} \times \prod_{i>r}L_{\Delta_i}
\]
 and 
 \begin{align*}
  \D^k(\pi(\a))&=\prod_{i=1}^{r}(L_{\Delta_i}+L_{\Delta_i^{(k)}})\times \prod_{i>r}L_{\Delta_i}\\
               &=\pi(\a)+\sum_{\Gamma\subseteq \a(k), \Gamma\neq \emptyset}\pi(\a_{\Gamma}).
\end{align*}
\end{proof}

Recall that in \cite[Notation 2.16]{Deng23}, we use the notation $\aR$ to denote the Grothendieck group
of all finite length unipotent representations of $GL_n(F)$ with $n=0,1,2, \ldots$ for a non-archimedean local field $F$.
Note that $\aR$ is a polynomial algebra endowed with a natural Hopf algebra structure (\cite[Proposition 2.17 and Corollary 2.18]{Deng23}).
\begin{lemma}\label{lem: 3.0.8}
 Let $\b\in S(\a)$, then $\pi(\a)-\pi(\b)\geq 0$ in $\mathcal{R}$.
\end{lemma}
 
\begin{proof}
 By choosing a maximal chain of multisegments between $\a$ and 
 $\b$, we can assume that 
\[
 \a=\{\Delta_{1}, \cdots, \Delta_{r}\},
 \]
 \[
 \b=(\a\backslash\{\Delta_{j}, \Delta_{k}\})\cup \{\Delta_{j}\cap \Delta_{k}, \Delta_{j}\cup \Delta_{k}\}.
\]
Then by \cite[section 4.6]{Z3},
\[ 
 \pi(\a)=\pi(\b)+L_{\Delta_{1}}\times \cdots \times \hat{L}_{\Delta_{j}}
 \times\cdots \times \hat{L}_{\Delta_{k}}\times \cdots\times  L_{\Delta_{r}}
\times L_{\{\Delta_{j}, \Delta_{k}\}}.
 \]
\end{proof}

\begin{prop}\label{prop: 7.4.2}
Let  
\addtocounter{theo}{1}
\begin{equation}\label{eqn-derivation-decomposition-standard}
\D^k(\pi(\a))=\sum_{\b}n_{\b, \a}L_{\b}.
\end{equation}
Then $n_{\b, \a}>0$ if and only if $\b\preceq_k \a$.  
\end{prop}

\begin{proof}
By \cite[Theorem 2.22]{Deng23} 
the coefficients $n_{\b, \a}$ in (\ref{eqn-derivation-decomposition-standard}) are
all non-negative (but notice that the $n_{\b, \a}$ are not the $n_{\b, \a}$ in loc.cit.).
Let $\b\preceq_k \a$, 
then by definition we have 
$\b\leq \a_\Gamma$ for some $\Gamma$. 
Therefore $m_{\b, \a_\Gamma}>0$ (cf. (\ref{eqn-decom-standard-irr})), now 
we have $n_{\b, \a}>0$ by  
(\ref{eq: 7.41}). Conversely, 
if $n_{\b, \a}>0$, then by 
(\ref{eq: 7.41}), 
 $\b\leq \a_\Gamma$ for some $\Gamma$. 
\end{proof}

\begin{cor}\label{cor: 7.4.3}
We have $\b\preceq_k \a$ 
 if and only if $\D^k(\pi(\a))-\pi(\b)\geq 0$
 in $\aR$. 
\end{cor}

\begin{proof}
By
Proposition \ref{prop: 7.4.2}  $\b\preceq_k \a$ implies
$\b\leq \a_\Gamma$ for some $\Gamma\subseteq \a(k)$. 
By Lemma \ref{lem: 3.0.8}
$\b\leq \a_\Gamma$ implies that 
$\pi(\a_\Gamma)-\pi(\b)\geq 0$ in $\aR$. 
Since $\D^k(\pi(\a))-\pi(\a_\Gamma)\geq 0$ by 
 (\ref{eq: 7.41}), we have $\D^k(\pi(\a))-\pi(\b)\geq 0$.
Conversely, if $\D^k(\pi(\a))-\pi(\b)\geq 0$, 
we have $n(\b, \a)>0$, hence $\b\preceq_k \a$
by Proposition \ref{prop: 7.4.2}.
\end{proof}

\begin{definition}\label{def: 1.2.10}
 We define for $\b\leq \a$,
 \[
  \ell(\b,\a)=\max_{n}\{n: \a=\b_{0}\geq \b_{1}\cdots \geq\b_{n}=\b\},
 \]
and $\ell(a)=\ell(\a_{\min}, \a)$.
\end{definition}

\begin{prop}\label{prop: 7.3.5}
For any $\b\preceq_k \a$, there exists $\c\in S(\a)$,
and some subset $\Gamma\subseteq \c(k)$,  
such that 
\[
 \b=\c_{\Gamma}.
\]
Conversely, if $\b=\c_{\Gamma}$ for some $\c\in S(\a)$,
 then $\b\preceq_k \a$.
\end{prop}

\begin{proof}
For the converse part, suppose $\c\neq \a$,
by (\ref{eq: 7.41}), we have $\D^k(\pi(\c))-\pi(\b)\geq 0$ in $\mathcal{R}$.
By Lemme \ref{lem: 3.0.8}, we 
know that $\pi(\a)-\pi(\c)\geq 0$ in $\aR$, hence $\D^k(\pi(\a))-\D^k(\pi(\c))\geq 0$
by  \cite[Theorem 2.22]{Deng23}. Therefore $n_{\b,\a}>0$. 
Hence we have $\b\preceq_k \a$.

For the direct part, suppose that 
$\b\preceq_k \a$, hence $\b<\a_{\Gamma_1}$ for some $\Gamma_1$. We prove by induction on $\ell(\b, \a_T)$.
If $\ell(\b, \a_{\Gamma_1})=0$, then $\b=\a_{\Gamma_1}$, we are done.
Now let $\b<\d\leq \a_{\Gamma_1}$ such that 
$\ell(\b, \d)=1$, by induction,
\[
 \d=\c'_{\Gamma_0}, 
\]
for some $\c'\in S(\a)$.
Note that by replacing $\c'$ by $\a$, we can assume that 
$\d=\a_{\Gamma_1}$ and $\ell(\b, \a_{\Gamma_1})=1$.
 
By definition $\b$ is obtained by 
applying the elementary operation to a pair of segments $\{\Delta\preceq \Delta'\}$
in $\a_T$. Now we set out to construct $\c$. 

\begin{itemize}
 \item If $\{\Delta, \Delta'\}\subseteq \a\setminus \{\Delta^{(k)}: \Delta\in \Gamma_1\}\subseteq \a$,
 let $\c$ be the multisegment obtained by applying the elementary operations to 
 $\{\Delta, \Delta'\}$. And we have 
 \[
  \b=\c_{\Gamma_1}.
 \]
\item If $\{\Delta, \Delta'\}\cap \{\Delta^{(k)}: \Delta\in \Gamma_1\}=\{\Delta'\}$, 
then $\{\Delta, \Delta'^{+}\}\in \a$
let $\c$ be the 
multisegment obtained by applying the elementary operations to 
 $\{\Delta, \Delta'^+\}$. Then 
 let 
 \[
  \Gamma=(\Gamma_1\setminus \{\Delta'^+\})\cup\{\Delta\cup \Delta'^+\}
 \]
and we have 
 \[
  \b=\c_{\Gamma}.
 \]
\item If $\{\Delta, \Delta'\}\cap \{\Delta^{(k)}: \Delta\in \Gamma_1\}=\{\Delta\}$, 
then $\{\Delta^{+}, \Delta'\}\in \a$
let $\c$ be the 
multisegment obtained by applying the elementary operations to 
 $\{\Delta^{+}, \Delta'\}$. Then let 
 \[
  \Gamma=(\Gamma_1\setminus \{\Delta^+\})\cup\{\Delta\cap \Delta'\}
 \]
 and we have 
  \[
  \b=\c_{\Gamma}.
 \]
\end{itemize}
Hence we are done.

\end{proof}

We use the notation $\O$ to denote the set of multisegments with support in $\Z$(cf. \cite[Definition 2.14]{Deng23}).
\begin{prop}
The relation $\preceq_k$ defines a 
poset structure on $\O$. 
\end{prop}

\begin{proof}
By definition we have $\a\preceq_k \a$ 
for any $\a\in \O$. 
Suppose $\a_1\preceq_k \a_2, \a_2\preceq_k \a_3$, 
we want to show that $\a_1\preceq_k \a_3$. 
By proposition \ref{prop: 7.3.5}, there exists $\c\in S(\a_2)$
and $\Gamma_{1}\subseteq \c(k)$
, such that 
\[
 \a_1=\c_{\Gamma_1}.
\]
 Note that 
by Corollary \ref{cor: 7.4.3},  
the fact $\a_2\preceq_k \a_3$ implies $\D^k(\pi(\a_3))-\pi(\a_2)\geq 0$.
Hence we have $n(\a_3, \c)>0$, therefore $\c\preceq_k \a_3$ by Proposition \ref{prop: 7.4.2}.  
In turn there exists a multisegment $\c'\in S(\a_3)$
and $\Gamma_2\subseteq \c'(k)$, such that 
\[
 \c=\c'_{\Gamma_2}.
\]
Since we have $\c(k)\subseteq \c'(k)$, we take 
\[
 \Gamma_3: =\Gamma_1\cup \Gamma_2\subseteq \c'(k).
\]
Now we get 
\[
 \a_1=\c'_{\Gamma_3},
\]
which implies $\a_1\preceq_k \a_3$ by Proposition \ref{prop: 7.3.5}. 
Finally, if $\a\preceq_k \b$ and $\b\preceq_k \a$, then by definition
we have $\a=\b$. 

\end{proof}

\begin{definition}\label{def-gamma-a-k}
We let 
\[
\Gamma(\a, k)=\{\b: \b\preceq_k \a\}.
\]
\end{definition}

\section{Canonical Basis and Quantum Algebras}\label{sec-recall-Lusztig}

In this section, we recall 
the results of Lusztig \cite{Lu} on canonical basis while keeping the notations of \cite{Deng23}
and draw the connection to the algebra $\mathcal{R}$ base on \cite{LNT}. No new results are
contained in this section.


\begin{definition}
Let $\N^{(\Z)}$ be the semi-group of sequences $(d_j)_{j\in \Z}$ of non negative integers 
which are zero for all but finitely many $j$.   
Let $\alpha_i$ be the element whose $i$-th term is 1 and other terms are zero.
\end{definition}

\begin{definition}
 We define a symmetric bilinear form on $\N^{(\Z)}$ given by 
 \begin{displaymath}
 (\alpha_i, \alpha_j)=
 \left\{\begin{array}{cc}
 2, &\text{ for } i=j;\\
 -1,  &\text{ for } |i-j|=1;\\
 0,   &\text{ otherwise }.
   \end{array}\right. 
 \end{displaymath}
\end{definition}

\begin{definition}
Let $q$ be an indeterminate and $\Q(q^{1/2})$ be the fractional field of $\Z[q^{1/2}]$.
Let $U_q^{\geq 0}$ be the $Q(q^{1/2})$-algebra generated by the elements $E_i$ and $K_i^{\pm 1}$ for $i\in \Z$
with the following relations:
\begin{align*}
K_iK_j=K_jK_i,~K_iK_i^{-1}=1;&\\
K_iE_i=q^{1/2(\alpha_i, \alpha_j)}E_iK_i ;&\\
E_iE_j=E_jE_i, &\text{ if } |i-j|>1;\\
E_i^2E_{j}-(q^{1/2}+q^{-1/2})E_iE_jE_i+E_jE_i^2=0, &\text{ if }|i-j|=1.
\end{align*}
and let $U^{+}$ be the subalgebra generated by 
the $E_i$'s. 
\end{definition}

\remk This is the $+$ part of the quantized enveloping algebra $U$ associated by Drinfeld
and Jimbo to the root system $A_{\infty}$ of $SL_{\infty}$. And for $q=1$, this specializes to the 
classical enveloping algebra of the nilpotent radical of a Borel subalgebra.

\begin{definition}\label{def: 7.2.4}
We define a new order on the set of segments $\Sigma$

\begin{displaymath} \left\{ \begin{array}{cc}
&[j, k]\lhd  [m,n], \text{ if } k< n,\\
&[j,k]\rhd  [m,n], \text{ if } j<m, n=k.
\end{array}\right. 
\end{displaymath}
We also denote $[j,k]\lhd [m, n]$ or $[j, k]=[m, n]$ by $\unlhd [m, n]$. 
\end{definition}
\remk This should be compared with \cite[Definition 2.5]{Deng23}.
\begin{lemma}
The algebra $U_q^+$ is $\N^{(\Z)}$-graded via the weight function $\wt(E_i)=\alpha_i$. 
Moreover, for a given weight $\alpha$, the homogeneous component of $U_q^+$ with weight $\alpha$ 
is of finite dimension, and its basis are naturally parametrized by the multisegments of the same 
weight.
\end{lemma}

\begin{proof}
Let $\a=\sum_{s=1}^rm_{i_s, j_s} [i_s, j_s]$ be a multisegment of weight $\alpha$ (cf. \cite[Definition 3.1]{Deng23}), note that 
here we identify the weight $\varphi_{[i]}$ with $\alpha_i$, and that 
\[
 [i_1, j_1]\unlhd \cdots \unlhd [i_r, j_r] ( \text{ cf. Definition \ref{def: 7.2.4}} )
\]
Then we associate to $\a$ the element
\[
 (E_{j_1}\cdots E_{i_1})\cdots (E_{j_r}\cdots E_{i_r}).
\] 
In this way we obtain the desired parametrization.
\end{proof}

\begin{notation}
For $x\in U^+$ be an element of degree $\alpha$, we will denote $\wt(x)=\alpha$. 
\end{notation}

\begin{example}
For $i\leq j$, let $\alpha_{ij}=\alpha_i+\cdots+\alpha_j$. Consider the homogeneous components 
of $U^+$ with weight $\alpha=2\alpha_{12}$, whose basis is given by 
 \[
  E_1E_2E_1E_2, ~E_1E_1E_2E_2. 
 \]
The element $ E_1E_2E_1E_2$ is parametrized by the 
multisegment $[1]+[1,2]+[2]$, while $E_1E_1E_2E_2$ is parametrized
by the multisegment $2[1]+2[2]$.
\end{example}

In \cite{Lu}, Lusztig has defined certain bases for $U_q^+$
associated to the orientations of a Dynkin diagram, called PBW( Poincar\'e-Birkhoff-Witt) basis, 
which
specializes to the classical PBW type bases.
Following \cite{LNT}, we describe the PBW-basis as follows.

\begin{definition}
We define 
\[
 E([i])=E_i, ~E([i, j])=[E_j[\cdots [E_{i+1}, E_i]_{q^{1/2}}\cdots ]_{q^{1/2}}  ]_{q^{1/2}},
\]
where $[x, y]_{q^{1/2}}=xy-q^{-1/2(\wt(x), \wt(y))}yx$.  
More generally, let 
$\a=\sum_s a_{i_s, j_s}[i_s, j_s]$ be a multisegment, such that 
\[
 [i_1, j_1]\unlhd \cdots \unlhd [i_r, j_r](  \text{ cf. Definition \ref{def: 7.2.4}}), 
\]
we define 
\[
 E(\a)=\frac{1}{\prod_s [a_{i_s,j_s}]_{q^{1/2}}!}E([i_1,j_1])^{a_{i_1, j_1}}\cdots E([i_r, j_r])^{a_{i_r,j_r}},
\]
here $[m]_{q^{1/2}}=\frac{q^{1/2m}-q^{-1/2m}}{q^{1/2}-q^{-1/2}}$ for $m\in \Z$ and 
$[m]_{q^{1/2}}!=[m]_{q^{1/2}}[m-1]_{q^{1/2}}\cdots [2]_{q^{1/2}}$.

\end{definition}

\begin{definition}
Let $x\mapsto \line{x}$ be the involution defined as the unique ring  automorphism of $U_q^+$ defined by 
\[
 \line{q^{1/2}}=q^{-1/2},~ \line{E_i}=E_i.
\]
\end{definition}

\begin{prop}(cf. \cite[Theorem 8.10]{Lu})
Let $\mathcal{L}:=\bigoplus_{\a\in \O}\Z[q^{1/2}]E(\a)\subseteq U_q^+$. Then 
there exists a unique $\Q(q^{1/2})$-basis $\{G(\a): \a\in \O\}$ of $U_q^+$ such 
that 
\[
 \line{G(\a)}=G(\a), ~ G(\a)=E(\a) \text{ modulo }q^{1/2}\mathcal{L}.
\]
This is called Lusztig's canonical basis. 
\end{prop}

Lusztig also gave a geometric description 
of his canonical basis in terms of 
the orbital varieties $\line{O}_{\a}$ \cite[Proposition 3.3]{Deng23}.

\begin{definition}
Let $\A$ be the group ring of $\line{\Q}^*_{\l}$ over $\Z$.
Let $\K_{\varphi}$ be the Grothendieck group 
over $\A$ of the category of constructible $G_{\varphi}$-equivariant
$\Q_{\l}$ sheaves over $E_{\varphi}$ (cf. \cite[Definition 3.2]{Deng23}),  considered as a variety over a 
finite field $\mathbb{F}_q$. 
\end{definition}

Given a weight function $\varphi$ \cite[Definition 3.2]{Deng23}, we use $S(\varphi)$ to denote the finite set of 
multisegments with weight $\varphi$. We also use $\mathcal{S}$ to denote the space of weight functions.

\begin{lemma}(cf. \cite[\S 9.4]{Lu})
The $\A$-module $\K_{\varphi}$ admits  a basis  $\{\gamma_{\a}: \a\in S(\varphi)\}$
indexed by the $G_{\varphi}$ orbits of $E_{\varphi}$, where $\gamma_{\a}$ corresponds 
to the constant  sheaf $\line{\Q}_{\l}$ on the orbit $O_{\a}$, extending by $0$ to the complement.
\end{lemma}

\begin{definition}\label{def: 7.4.12}
Let $\varphi=\varphi_1+\varphi_2\in \mathcal{S}$.
We define a diagram of varieties
\addtocounter{theo}{1}
\begin{equation}\label{di: 7.3}
\xymatrix
{
E_{\varphi_1}\times E_{\varphi_2}& E'{\ar[l]_{\hspace{0.8cm}\beta}}\ar[r]^{\beta'}&E''\ar[r]^{\beta''}&E_{\varphi},
}
\end{equation}
where 
\begin{align*}
 E'':=&\{(T, W): W=\bigoplus W_i,~ W_i \subseteq V_{\varphi, i},~ T(W_i)\subseteq W_{i+1}, ~\dim(W_i)=\varphi_2(i)\},\\
 E':=& \{(T, W, \mu, \mu'): (T, W)\in E'',~ \mu: W\simeq V_{\varphi_2}, ~\mu': V_{\varphi}/W\simeq V_{\varphi_1}\},
\end{align*}
and 
\[
 \beta''((T, W))=W, ~ \beta'((T, W, \mu, \mu'))=(T, W), ~\beta((T, W, \mu, \mu'))=(T_1, T_2),
\]
such that 
\[
 T_1=\mu' \circ T\circ \mu'^{-1}, ~ T_2=\mu\circ T\circ \mu^{-1}.
\]
\end{definition}

\begin{prop}(cf. \cite[\S 6.1]{Lu})
The group $G_{\varphi}\times G_{\varphi_1}\times G_{\varphi_2}$ 
acts naturally on the varieties in the diagram (\ref{di: 7.3})
with $G_{\varphi}$ acting trivially on $E_{\varphi_1}\times E_{\varphi_2}$
and $G_{\varphi_1}\times G_{\varphi_2}$ acting trivially on $E_{\varphi}$.
And all the maps there are compatible with such actions. Moreover,
we have 
\begin{description}
 \item[(1)]The morphism $\beta'$ is 
 a principle $G_{\varphi_1}\times G_{\varphi_2}$-fibration.
 \item[(2)]The morphism $\beta$ is a locally trivial trivial 
 fibration with smooth connected fibers.
 \item[(3)]The morphism $\beta''$ is proper.
\end{description} 
\end{prop}

\begin{example}\label{ex: 7.4.14}
Let $\varphi_1=\chi_{1}$ (the singleton supported at the integer $1\in \O$) and $\varphi_{2}=\chi_2$. Then $\varphi=\chi_1+\chi_2$
and 
\[
 E_{\varphi_1}=E_{\varphi_2}=0, ~ E_{\varphi}=\line{\mathbb{F}}_{q}.
\]
Moreover, we have 
\[
 E''=\{(T, W): W=V_{\varphi_2}, T\in \line{\mathbb{F}}_q\}\simeq \line{\mathbb{F}}_q, 
\]
and 
\[
 E'=\{(T, W, \mu, \mu'): (T, W)\in E'', \mu, \mu'\in \line{\mathbb{F}}_q^{\times}\}
 \simeq \line{\mathbb{F}}_q\times (\line{\mathbb{F}}_q^{\times})^2.
\]
\end{example}

\begin{cor}(cf. \cite[\S 9.5]{Lu})\label{cor: 7.4.15}
Let $\a\in S(\varphi_1), \a'\in S(\varphi_2)$ and $\varphi=\varphi_1+\varphi_2$ be a weight function.
There exists a simple perverse sheaf( up to shift )  $\mathcal{P}$ on $E_{\varphi}$ such that 
\[
  \beta^*(IC(\line{O}_{\a})\otimes IC(\line{O}_{\a'}))=\beta'^{*}(\mathcal{P}). 
\]
\end{cor}

\begin{example}
As in Example \ref{ex: 7.4.14}, 
let $\a=\{[1]\}, \a'=\{[2]\}$, then 
\[
 IC(\line{O}_{\a})=\line{\Q}_{\l}, ~IC(\line{O}_{\a'})=\line{\Q}_{\l}.
\]
Hence if we let 
\[
 \mathcal{P}=\line{\Q}_{\l},
\]
then 
\[
 \beta^{*}(IC(\line{O}_{\a})\otimes IC(\line{O}_{\a'}))=\beta''^{*}(\mathcal{P}).
\]

\end{example}

\begin{definition}
 We define a multiplication  
\[
  IC(\line{O}_{\a})\star IC(\line{O}_{\a'})=\beta''_{*}(\mathcal{P}).
 \]
\end{definition}

\begin{example}
As in the Example \ref{ex: 7.4.14}, we have 
\[
 IC(\line{O}_{\a})\star IC(\line{O}_{\a'})=\beta''_{*}(\mathcal{P})=IC(E_{\varphi}),
\]
note that here $\beta''$ is an isomorphism. 
\end{example}

\begin{prop}(cf. \cite[\S 9.5]{Lu})
 Let $\a\in S(\varphi_1), \a'\in S(\varphi_2)$ $\varphi=\varphi_1+\varphi_2$ be a weight function.
We associate to the intersection cohomology complex $IC(\line{O}_{\a})$
\[
 \tilde{\gamma}_{\a}=\sum_{\b\geq \a}p_{\b, \a}(q)\gamma_{\b},
\]
where $p_{\b, \a}(q)$ is the formal alternative sum 
of eigenvalues of the Frobenius map on the stalks of the cohomology sheaves 
of $IC(\line{O}_{\a})$ at any $\mathbb{F}_q$ rational point of $O_{\b}$. 
Moreover, the multiplication $\star$ gives a  $\A$-bilinear map 
\[
 \K_{\varphi_1}\times \K_{\varphi_2}\rightarrow \K_{\varphi},
\]
which defines an associative algebra structure over $\K=\bigoplus_{\varphi}\K_{\varphi}$. 
\end{prop}

\begin{prop}(\cite[Proposition 9.8 and Theorem 9.13]{Lu} )\label{prop: 7.4.20}

\begin{itemize}
 \item
The elements $\gamma_{i}:=\gamma_{[i]}$ for all $i\in \Z$
generate the algebra $\K$ over $\A$.
 
 \item Let $U^{\geq 0}_{\A}=U_q^{\geq 0}\otimes_{\Z} \A$. 
 Then we have a unique $\A$-algebra morphism 
 $\Gamma: \K\rightarrow U_{\A}^{\geq 0}$ such that 
 \[
  \Gamma(\gamma_{j})=K_j^{-j}E_{j}; 
 \]
for all $j\in \Z$.  Moreover, for $\varphi\in \mathcal{S}$, let
\begin{equation}\label{eqn-weight-fractor-quantum}
 d(\varphi)=\sum_{i\in \Z}(\varphi(i)-1)\varphi(i)/2- \sum_{i\in \Z}\varphi(i)\varphi(i+1).
\end{equation}
Then there is an $\A$-linear map $\Theta: K_{\varphi}\rightarrow U_{\A}^{+}$, such that 
\[
 \Gamma(\xi)=q^{1/2d(\varphi)}K(\varphi)\Theta(\xi),
\]
where $K(\varphi)=\prod_{i\in \Z}K_{i}^{-i\varphi(i)}$. 
 \item 
 We have 
 \[
 \Gamma(\gamma_{\c})=
 q^{1/2(r-\delta_{\c})}K(\varphi_{\c})E(\c), 
 \]
where 
\[
 r=\sum_{i}\varphi_{\c}(i)(\varphi_{\c}(i)-1)(2i-1)/2-\sum_{i}i\varphi_{\c}(i-1)\varphi_{\c}(i), 
\]
and $\delta_{\c}$ is the co-dimension of the orbit $O_{\c}$ in $E_{\varphi_{\c}}$.

 \item We have 
 \[
  \Theta(\gamma_{\a})=q^{1/2\dim(O_{\a})}E(\a),~ \Theta(\tilde{\gamma}_{\a})=q^{1/2\dim(O_{\a})}G(\a).
 \]
Hence 
\[
 G(\a)=\sum_{\b\geq \a}P_{\b, \a}(q)E(\b). 
\]

\end{itemize}
\end{prop}

\begin{prop}\label{prop: 7.4.21}\cite[\S 4.3]{LTV}
The canonical basis of $U_q^+$ are almost orthogonal with 
respect to a scalar product introduced by Kashiwara \cite{K}, 
which are given by 
\[
 (E(\a), E(\b))=\frac{(1-q)^{\deg(\a)}}{\prod_{i\leq j}h_{a_{ij}}(q)}\delta_{\a, \b},
\]
where $\a=\sum_{i\leq j}a_{ij}[i,j]$,
$h_{k}(z)=(1-z)\cdots (1-z^k)$ and $\delta$ is the Kronecker symbol.
And we have 
\[
 (G(\a), G(\b))=\delta_{\a, \b}\text{  mod }q^{1/2}\A.
\]

\end{prop}

\begin{notation}
We denote by $\{E^*(\a)\}$ and $\{G^*(\a)\}$ 
the dual basis of $\{E(\a)\}$ and $\{G(\a)\}$ with 
respect to the Kashiwara scalar product.
\end{notation}

\remk Note that $\{G^*(\a)\}$  is referred as the dual canonical basis.

\begin{prop}(cf. \cite[\S 3.4]{LNT})
Let $\a=\sum_{i\leq j}a_{ij}[i,j]$. Then 
 \begin{itemize}
  \item We have 
 \[
  E^*(\a)=\frac{\prod\limits_{i\leq j}h_{a_{ij}}(q)}{(1-q)^{\deg(\a)}}E(\a)=
  \overrightarrow{\prod\limits_{ij}}q^{1/2\binom{a_{ij}}{2}}E^*([i,j])^{a_{ij}},
 \]
here the product is taken with respect to the order $\leq $. 

\item And 
  \[
   E^*(\a)=\sum_{\b\leq \a}P_{\a, \b}(q)G^*(\b).
  \]

 \end{itemize} 
\end{prop}

\begin{example}
 Let $\a=[1]+[2]$. Then 
 \[
  E^*(\a)=E([1])E([2])=E(\a),
 \]
 and 
 \[
  G^*([1,2])=E^*([1,2]),
 \]
\[
 G^*(\a)=E^*(\a)-q^{1/2}E^*([1,2]).
\]
\end{example}

Finally, we establish the relation of 
between the algebras $\aR$ and $U^+$.

\begin{definition}
 Let $B$ be the polynomial algebra generated by 
 the set of coordinate functions $\{t_{ij}: i< j\}$.
 Following \cite[\S 2.6]{LNT}, we write $t_{ii}=1$, 
 $t_{ij}=0$ if $i>j$, and index the non-trivial $t_{i,j}$'s 
 by segments, namely, $t_{[ij]}=t_{i,j-1}$ for $i<j$. 
\end{definition}

\begin{prop}\label{prop: 7.4.25} \cite[Corollary 2.18]{Deng23},
 We have an algebra isomorphism $\phi: B\simeq \aR$ by identifying 
 $t_{[ij]}$ with $L_{[ij]}$ for all $i<j$.  
\end{prop}

\begin{definition}(\cite{BZ2},  \cite[\S 3.5]{LNT})
 Let $B_q$ be the quantum analogue of $B$ generate by 
 $\{T_{ij}: i<j\}$, where $T_{ij}$ is considered as 
 the $q$-analogue of $t_{ij}$. Also, we write $T_{ii}=1$
 and $T_{ij}=0$ if $i>j$. And we will indexed the non-trivial 
 $T_{ij}$ by $T_{[i,j-1]}$. The generators $T_s$'s satisfies 
 the following relations. Let $s>s'$ be two segments.
 Then 
 \begin{displaymath}
 T_{s'}T_{s}=
 \left\{ \begin{array}{cc}
    q^{- 1/2(\wt(s'), \wt(s))}T_sT_{s'}+(q^{-1/2}-q^{1/2})T_{s\cap s'}T_{s\cup s'},
    &\text{  if } s \text{ and } s' \text{ are linked, }\\
    q^{-1/2(\wt(s'), \wt(s))}T_sT_{s'}, \text{ otherwise }.
   \end{array} \right.
\end{displaymath}
\end{definition}

\begin{prop}\cite[\S 3.5]{LNT} \label{prop: 7.4.23}
There exist an algebra isomorphic morphism 
\[
 \iota: U_q^+\rightarrow B_q, 
\]
given by $\iota(E^*([i,j]))=T_{[i,j]}$. Moreover, for $\a=\sum_{i\leq j}a_{ij}[i,j]$,
we have 
\[
 \iota(E^*(\a))=\overrightarrow{\prod_{i\leq j}}q^{1/2\binom{a_{ij}}{2}}T_{[i,j]}^{a_{ij}},
\]
here the multiplication is taken with respect to the order $<$.
\end{prop}

\begin{example}
Let $\a=[1]+[2]$, then 
\[
 \iota(E^*(\a))=T_{[1]}T_{[2]}.
\]

\end{example}

\begin{prop}\cite[Theorem 12]{LNT}
By specializing at $q=1$, the dual canonical basis $\{G^*(\a): \a\in \O\}$ gives rise
to a well defined basis for $B$, denoted by $\{g^*(\a): \a\in \O\}$. Moreover, 
the morphism $\phi$ (cf. Proposition \ref{prop: 7.4.25}) sends $g^*(\a)$ to $L_{\a}$ for all $\a\in \O$. 
\end{prop}

\section{Partial BZ operator and Poincar\'e's series}

In this section we will deduce a geometric description for the 
partial BZ operator, using results of last section.

\begin{definition}\cite[\S 2.2]{L}
Kashiwara \cite{K}  introduced some $q$-derivations $E'_i$ in $\End(U_q^{+})$
for all $i\in \Z$ satisfying 
\[
 E_i'(E_j)=\delta_{ij}, ~E_i'(uv)=E_i'(u)v+q^{-1/2(\alpha_i, \wt(u))}uE_i'(v). 
\]
\end{definition}

\begin{example}
Simple calculation shows that  
\[
E_i'(E([j,k]))=\delta_{i, k}(1-q)E([j, k-1]), 
\]
 by taking dual, we get 
 \[
  E_i'(E^*([j,k]))=\delta_{i, k}E^*([j, k-1]), 
 \] 
\end{example}

\begin{prop}
We have 
\[
 (E_i'(u), v)=(u, E_iv),
\]
where $( , )$ is the scalar product introduced in 
Proposition \ref{prop: 7.4.25}. 
\end{prop}

Note that by identifying 
the algebra $U_q^{+}$ and 
$B_q$ via $\iota$, we get a version of  
$q$-derivations in $\End(B_q)$. 

\begin{definition}\label{def: 7.3.4}
By specializing at $q=1$, 
the $q$ derivation $E_i'$ gives a derivation $e_i'$ of the algebra $B$ by  
 \[
  e_i'(t_{[jk]})=\delta_{ik}t_{[j,k-1]}, 
  ~ e_i'(uv)=e_i'(u)v+ue_i'(v). 
 \] 
\end{definition}

\begin{prop}\label{prop: 7.6.5}
Let 
\[
 D^i: =\sum_{n=0}^{\infty}\frac{1}{n!}{e_i'}^n.
\]
Then the morphism $D^i: B\rightarrow B$ is an 
algebraic morphism. Moreover, if we identify 
the algebras $\mathcal{R}$ and $B$ via $\phi$ (cf. Proposition \ref{prop: 7.4.25}), then 
the morphism $D^i$ coincides with  the partial BZ operator
$\D^i$.  

\end{prop}

\begin{proof}
For $n\in \N$, we have 
\[
 e'^n(uv)=\sum_{r+s=n}\binom{n}{r}e_{i}'^r(u)e_{i}'^s(v), 
\]
therefore 
\[
 D^i(uv)=\sum_{n=0}^{\infty}\frac{1}{n!}\sum_{r+s=n}\binom{n}{r}e_{i}'^r(u)e_{i}'^s(v)=D^i(u)D^i(v). 
\]
Finally, to show that $D^i$ and $\D^i$ coincides, 
it suffices to prove that 
\[
 \phi\circ D^i(t_{[j,k]})=\D^i\circ \phi(t_{[j,k]}),
\]
but we have 
\[
 D^i(t_{[j,k]}=t_{[j,k]}+\delta_{i,k}t_{[j, k-1]},
\]
 and 
 \[
  \D^i(L_{[j,k]})=L_{j,k}+\delta_{i,k}L_{[j, k-1]}.
 \]
Therefore, we have 
\[
 \phi\circ D^i(t_{[j,k]})=\D^i\circ \phi(t_{[j,k]}).
\]  
\end{proof}

\remk Without specializing at $q=1$, 
the operator $D^i$ is not an algebraic morphism.  
To get an algebraic morphism at the level of 
$U_q^+$, one should consider not only 
the summation of the iteration of $e_i'$'s
but all the derivations, which 
gives rise to an embedding into the  quantum shuffle algebras, cf. \cite{L}.

Next we show how to determine $\D^i(L_{\a})$
by the algebra $\K$ of Lusztig.

\begin{lemma}\label{lem: 7.6.6}
Let $n\in \N$, and $ \d\in \O$. Then we have  
\[
 E_i^nG(\b)=\sum_{\d}(E_i^nG(\b), E^*(\d))E(\d)=\sum_{\d}(G(\b), E_i'^nE^*(\d))E(\d),
\]
where $( , )$ is the Kashiwara scalar product. 
Moreover, for each $\b$ such that 
$(G(\b), E_i'^nE^*(\d))\neq 0$, we have 
\[
 \wt(\d)=\wt(\b)+n\alpha_i.
\]

\end{lemma}

\begin{proof}
This is by definition. 
\end{proof}

\begin{cor}
Let $\b\preceq_k \d$ 
such that $wt(\d)=\wt(\b)+n\alpha_i$. 
Then $L_{\b}$ appears as a factor of $\frac{1}{r!}e_i'^r(\pi(\d))$
if and only if $r=n$. 

\end{cor}

\begin{proof}
For each $\b\preceq_k \d$, 
the representation $L_{\b}$ is a factor of $\D^i(\pi(\d))$.  
Now by Proposition \ref{prop: 7.6.5}, 
$\D^i=\sum_{r}\frac{1}{r!}e_i'^r$, moreover, 
by Lemma \ref{lem: 7.6.6}, factors of $\frac{1}{r!}e_i'^r(\pi(\d))$
always have weight $\wt(\d)-r\alpha_i$. 
Therefore we are done. 
\end{proof}

\begin{teo}\label{prop: 7.3.8}
Let $\b\preceq_k \a$, then
there exists $\c\in S(\a)$ such that $\c=\b+\l[k]$. Then 
\[
 n_{\b, \a}=\sum_{i}\dim\mathcal{H}^{2i}( IC(\line{O}_{\l[k]})\star IC(\line{O}_{\b}))_{\a}.
\] 
\end{teo}

\begin{proof}
Note that by Proposition \ref{prop: 7.4.20}, we have  
\begin{align*}
 \Gamma(\tilde{\gamma}_{\l[k]}\star \tilde{\gamma}_{\b})&=
 \Gamma(\tilde{\gamma}_{\l[k]})\Gamma(\tilde{\gamma}_{\b})\\
 &=q^{1/2(d(\varphi_{\l[k]})+d(\varphi_{\b}))}K(\varphi_{\l[k]})K(\varphi_{\b})\Theta(\tilde{\gamma}_{\l[k]})
 \Theta(\tilde{\gamma}_{\b})\\
 &=q^{1/2(d(\varphi_{\l[k]})+d(\varphi_{\b}))}K(\varphi_{\l[k]}+\varphi_{\b})q^{1/2(\dim(O_{\l[k]})+\dim(O_{\b}))}G(\l[k])
 G(\b).
 \end{align*}
Since we have 
\[
d(\varphi_{\l[k]})=\dim(O_{\l[k]})=0, ~G(\l[k])=E(\l[k])=\frac{1}{[\l]_{q^{1/2}}!}E_{k}^{\l},
~\varphi_{\l[k]}+\varphi_{\b}=\varphi_{\a},
\]
so 
\[
 \Gamma(\tilde{\gamma}_{\l[k]}\star \tilde{\gamma}_{\b})=
 \frac{1}{[\l]_{q^{1/2}}!}q^{1/2(d(\varphi_{\b})+\dim(O_{\b}))}K(\varphi_{\a})E_{k}^{\l}G(\b).
\]
And 
\begin{align*}
 \Gamma(\gamma_{\d})&=q^{1/2d(\varphi_{\d})}K(\varphi_{\d})\Theta(\gamma_{\d})\\
 &=q^{1/2(d(\varphi_{\d})+\dim(O_{\d}))}K(\varphi_{\d})E(\d).
\end{align*}
Now write
\[
 \tilde{\gamma}_{\l[k]}\star \tilde{\gamma}_{\b}=\sum_{\b\preceq_k \d, \varphi_{\d}=\varphi_{\a}}p_{\d, \b}(q)\gamma_{\d},
\]
with 
\[
 p_{\d, \b}(q)=\sum_{i}q^{i}\mathcal{H}^{2i}(IC(\line{O}_{\l[k]})\star IC(\line{O}_{\b}))_{\d}.
\]
Applying $\Gamma$ gives 
\begin{multline*}
 \frac{1}{[\l]_{q^{1/2}}!}q^{1/2(d(\varphi_{\b})+\dim(O_{\b}))}K(\varphi_{\a})E_{k}^{\l}G(\b)= \\
 \sum_{\b\preceq_k \d, \varphi_{\d}=\varphi_{\a}}p_{\d, \b}(q)q^{1/2(d(\varphi_{\d})+\dim(O_{\d}))}K(\varphi_{\d})E(\d).
\end{multline*}
Hence 
\[
 E_{k}^{\l}G(\b)=[\l]_{q^{1/2}}!\sum_{\b\preceq_k \d, \varphi_{\d}=\varphi_{\a}}p_{\d, \b}(q)
 q^{1/2(d(\varphi_{\d})+\dim(O_{\d})-d(\varphi_{\b})-\dim(O_{\b}))}E(\d),
\]
now compare with Lemma \ref{lem: 7.6.6},
we get 
\[
 (G(\b), E_i'^nE^*(\d))=[\l]_{q^{1/2}}!p_{\d, \b}(q)
 q^{1/2(d(\varphi_{\d})+\dim(O_{\d})-d(\varphi_{\b})-\dim(O_{\b}))}.
\] 
Finally, 
we write 
\[
 \frac{1}{[\l]_{q^{1/2}}!}E_i'^nE^*(\d)=\sum_{\b}n_{\b, \d}(q)G^{*}(\b),
\]
by applying the scalar product, we get 
\[
 n_{\b, \d}(q)=(G(\b), \frac{1}{[\l]_{q^{1/2}}!}E_i'^nE^*(\d))
 =p_{\d, \b}(q)
 q^{1/2(d(\varphi_{\d})+\dim(O_{\d})-d(\varphi_{\b})-\dim(O_{\b}))}.
\]
Hence, by specializing at $q=1$, we have 
\[
 n_{\b, \d}=p_{\d, \b}(1).
\]
Now take $\d=\a$, we get the formula 
in our theorem.
\end{proof}

\section{A formula for  Lusztig's product}\label{sec-Lusztig-prod}

In this section we find a geometric way to calculate Lusztig's product in special case, which allows us 
to determine the partial BZ operator in the following sections.

\begin{definition}
Let $k\in \Z$.
We say that  $\a$ satisfies the assumption 
$(\A_k)$ if it satisfies the following conditions
\footnote{Since here we only work with the partial BZ operator 
$\D^k$ with $k\in \Z$, 
for every multisegment, we can always use the reduction 
method to increase the length of segments from the left, so 
that at some point we 
arrive at the situation of our assumption $(\A_k)$, therefore we do not lose the 
generality. For a more precise description of such a process, we refer to \cite[\S 6.2]{Deng23}}

\begin{description}
 \item[(1)] We have 
\[
 \max\{b(\Delta): \Delta\in \a\}+1<\min\{e(\Delta): \Delta\in \a\}.
\]
\item[(2)] 
Moreover, we have 
$\varphi_{e(\a)}(k)\neq 0$ and $\varphi_{e(\a)}(k+1)=0$.

\end{description}

\end{definition}

\begin{lemma}\label{lem: 7.6.1}
Let 
$\a$ be a multisegment satisfying 
the assumption $(\A_k)$. 
Then $\a$ is of parabolic type. Moreover, 
The set $S(\varphi_{\a})$ contains a unique maximal
element satisfying the assumption $(\A_k)$, denoted by $\a_{\Id}$. 
\end{lemma}

\begin{proof}
Let $b(\a)=\{k_1\leq \cdots\leq k_r\}$,
$e(\a)=\{\ell_1\leq \cdots \leq \ell_r\}$. 

Then by \cite[Proposition 3.32]{Deng24b}, 
there exists an element 
$w\in S_r^{J_1(\a), J_2(\a)}$, such that 
\[
 \a=\sum_{j}[k_j, \ell_{w(j)}].
\]
Let 
\[
 \a_{\Id}=\sum_{j}[k_j, \ell_j].
\]
By \cite[Proposition 3.28]{Deng24b} (The proposition is stated without proof but a proof can be given exactly as Proposition 3.18 in loc.cit..),  
$\a\leq \a_{\Id}$. 
Finally, $\a_{\Id}$ depends only 
on $b(\a)$ and $e(\a)$, not on $\a$, 
which shows that $\a_{\Id}$ is the maximal element
in $S(\varphi_{\a})$ satisfying the assumption $(\A_k)$. 
\end{proof}

Recall that in \cite[Definition 5.4]{Deng23} we introduced a subset $\tilde{S}(\a)_k$ of $S(\a)$ and further introduced a subset $S(\a)_k$
of $\tilde{S}(\a)_k$  in \cite[Definition 5.9]{Deng23}.

\begin{lemma}\label{lem: 7.6.2}
Suppose that $\a$ is a multisegment satisfying the hypothesis $(\A_k)$, then 
 
 \begin{description}
  \item[(1)] $\tilde{S}(\a)_k=S(\a)$,

\item[(2)] we have 
 \[
  X_{\a}^k=Y_{\a}=\coprod_{\c\in S(\a)}O_{\c}, 
 \]
  \end{description}
 where the variety $X_{\a}^{k}$ is defined in \cite[Definition 5.20]{Deng23}, which is a suitable sub-variety of $E_{\varphi}$ and
 $Y_{\a}$ is defined to be the union of all orbits indexed by elements of  $S(\a)$.
\end{lemma}

\begin{proof}
Note that by assumption  
\[
 \max\{b(\Delta): \Delta\in \a\}<\min\{e(\Delta): \Delta\in \a\}.
\]
This ensures that for any $\c\in S(\a)$, 
we have $\varphi_{e(\c)}(k)=\varphi_{e(\a)}(k)$,
hence by definition $\c\in \tilde{S}(\a)_{k}$. 
This proves (1), and (2) follows from (1).
\end{proof}

\begin{lemma}\label{lem: 7.6.3}
Let $\a$ be a multisegment 
satisfying the assumption $(\A_k)$
 and $\a=\a_{\Id}$. 
 Let $\ell\in \N$ such that $\ell\leq \varphi_{e(\a)}(k)$
 and   $\varphi\in \mathcal{S}$ such that 
 \[
  \varphi+\ell \chi_{[k]}=\varphi_{\a}.
 \]
Then for $\b\in S(\varphi)$, we have 
$\b\preceq_k \a$ if and only if $\b^{(k)}\leq \a^{(k)}$ (cf. Definition \ref{def: 2.2.3})
and $\varphi_{e(\b)}(k-1)=\ell+\varphi_{e(\a)}(k-1)$.
 
\end{lemma}

\begin{proof} 
Let 
$\b\in S(\varphi)$ such that 
$\b\preceq_k \a$, then 
by Proposition \ref{prop: 7.3.5}
$\b=\c_{\Gamma} $ for some 
$\c\in S(\a)$ and $\Gamma\subseteq \c(k)$.
Therefore 
\[
 \b^{(k)}=\c^{(k)}\leq \a^{(k)}
\]
by Lemma \ref{lem: 7.6.2}.  By definition of 
$\c_{\Gamma}$, 
\[
 \varphi_{e(\b)}(k-1)=\ell+\varphi_{e(\c)}(k-1).
\]
Now applying the fact that 
$\a$ satisfies the assumption $(\A_k)$, we deduce that 
\[
 \varphi_{e(\c)}(k-1)=\varphi_{e(\a)}(k-1).
\]
Conversely, let $\b\in S(\varphi)$ be a multisegment  such that 
$\b^{(k)}\leq \a^{(k)}$ and 
$\varphi_{e(\b)}(k-1)=\ell+\varphi_{e(\a)}(k-1)$.

We deduce from $\b^{(k)}\leq \a^{(k)}$ that 
\[
 \b\leq \a^{(k)}+\varphi_{e(\b)}(k)[k],
\]
from which we obtain
\[
 \varphi_{\b}=\varphi_{\a^{(k)}}+\varphi_{e(\b)}(k)\chi_{[k]}.
\]
By assumption
\[
 \varphi_{\b}+\ell\chi_{[k]}=\varphi_{\a}.
\]
Combining with the formula
\[
 \varphi_{\a}=\varphi_{\a^{(k)}}+\varphi_{e(\a)}(k)\chi_{[k]},
\]
we have 
\[
 \varphi_{e(\a)}(k)=\varphi_{e(b)}(k)+\ell.
\]
Now that for any 
$\Delta\in \a$, if $e(\Delta)=k$, 
then $b(\a)\leq k-1$. Therefore we have 
\[
 \varphi_{e(\a^{(k)})}(k-1)=\varphi_{e(\a)}(k-1)+\varphi_{e(\a)}(k).
\]
Applying the formula
$\varphi_{e(\b)}(k-1)=\ell+\varphi_{e(\a)}(k-1), ~\varphi_{e(\b^{(k)})}(k-1)=\varphi_{e(\a^{(k)})}(k-1)$, we get 
\[
  \varphi_{e(\b^{(k)})}(k-1)=\varphi_{e(\b)}(k-1)+\varphi_{e(\b)}(k).
\]
Such a formula implies that for $\Delta\in \b$, 
if $e(\Delta)=k$, then 
$b(\Delta)\leq k-1$. 

Let $b(\a)=\{k_1\leq \cdots\leq k_r\}$,
$e(\a)=\{\ell_1\leq \cdots \leq \ell_r\}$. 
The assumption that $\a=\a_{\Id}$ implies that 
\[
 \a=\sum_i [k_i, \ell_i]
\]
Suppose that 
\[
 \a(k)=\{[k_i,\ell_i]: i_0\leq i\leq i_1\}.
\]
Take 
$\Gamma=\{[k_i, \ell_i]: i_0+\ell\leq i\leq i_1\}$ and 
\[
 \a'=\a_{\Gamma}.
\]
Then $\a'\preceq_k \a$.  
Note that $\a'$ is a multisegment of 
parabolic type which 
corresponds to the identify 
in some symmetric group, 
cf. \cite[Notation 3.25]{Deng24b}. 
Finally, \cite[Proposition 3.28]{Deng24b}
implies that $\b\in S(\a')$.

\end{proof}

Recall that the variety $X_{\a}^k$ in Lemma \ref{lem: 7.6.2} admits a fibration $\alpha$ over 
the Grassmanian $Gr(\ell_{\a,k}, V_{\varphi})$ (cf. \cite[Proposition 5.28]{Deng23}) with
\[
\ell_{\a, k}=\sharp\{\Delta\in \a: e(\Delta)=k\}, (\text{ cf. \cite[Notation 5.12]{Deng23}}).
\]
For fixed $W\in Gr(\ell_{\a,k}, V_{\varphi})$, 
let $(X_{\a}^{k})_{W}$ be the fiber of $\alpha$ over $W$ (cf. \cite[Notation 5.30]{Deng23}).
By 
\cite[Proposition 5.35]{Deng23}, 
we have an open immersion
\[
 \tau_{W}:  (X_{\a}^{k})_{W}\rightarrow 
 (Z^{k, \a})_{W}\times \Hom(V_{\varphi, k-1}, W).
\]
Here $Z^{k, \a}$ is a sub-variety of a certain fiber bundles $\tilde{Z}^k$ over the Grassmanian $Gr(\ell_{\a,k}, V_{\varphi})$ (cf. \cite[Proposition 5.24]{Deng23}) and $(Z^{k, \a})_W$ is the fiber over $W$.
\begin{lemma}\label{lem: 7.7.2}
Assume that $\a$ is 
a multisegment satisfying $(\A_k)$.
Let $r\leq \varphi_{e(\a)}(k)$ and 
$\d=\a+r[k+1]$. Then 
we have 
$X_{\d}^{k+1}=Y_{\d}$ and for a 
fixed subspace $W$ of $V_{\varphi_{\d}, k+1}$
of dimension $r$,
the open immersion 
\[
 \tau_W: (X_{\d}^{k+1})_{W}\rightarrow (Z^{k+1, \d})_W\times \Hom(V_{\varphi_{\d}, k}, W)
\]
is an isomorphism.
   
\end{lemma}

\begin{proof}
Note that our assumption on $\a$ 
ensures that $X_{\d}^{k+1}=Y_{\d}$ since 
we have $\d_{\min}\in \tilde{S}(\d)_{k+1}$ (cf. \cite[Notation 5.11]{Deng23}). 
It suffices to show that $\tau_W$ is surjective.
Let $(T^{(k)}, T_0)\in (Z^{k+1, \d})_W\times \Hom(V_{\varphi_{\d}, k}, W)$,
by fixing a splitting $V_{\varphi_{\d}, k+1}=W\oplus V_{\varphi_{\d}, k+1}/W$, we define 
 \[
  T'|_{V_{\varphi_{\d}, k}}=T_{0}\oplus T^{(k+1)}|_{V_{\varphi_{\d}, k}},
 \]
\[
 T'|_{V_{\varphi_{\d}, k+1}}=T^{(k+1)}|_{V_{\varphi_{\d}, k+1}/W}\circ p_{W},
\]
\[
 T'|_{V_{\varphi_{\d}, i}}=T^{(k+1)}, \text{ for } i\neq k, k+1,
\]
where $p_w: V_{\varphi_{\d}}\rightarrow  V_{\varphi_{\d}, k}/W$ is the 
canonical projection.
Then we have $T'\in Y_{\d}$ hence $T'\in (X_{\d}^{k+1})_{W}$. 
Now since by construction we have $\tau_W(T')=(T^{(k)}, T_0)$, we are done.
\end{proof}

\begin{definition}\label{def: 7.4.6}
Assume that $\a$ is 
a multisegment satisfying $(\A_k)$
and $\d=\a+r[k+1]$ for some $r\leq \varphi_{e(\a)}(k)$.
Let $\mathfrak{X}_{\d}$ be the open sub-variety of 
$X_{\d}^{k+1}$ consisting of those orbits 
$O_{\c}$ with $\c\in S(\d)$, such that  
 $\varphi_{e(\c)}(k)+r=\varphi_{e(\a)}(k)$. 
\end{definition}

\begin{definition}
Let $V$ be a vector space and $\ell_1<\ell_2<\dim(V)$ be two integers.
We define 
\[
 Gr(\ell_1, \ell_2, V)=\{(U_1, U_2): U_1\subseteq U_2\subseteq V, \dim(U_1)=\ell_1, \dim(U_2)= \ell_2\}.
\]

\end{definition}

\begin{definition}\label{def-open-subset-Ea}
Let $\ell$ be an integer and 
$\a$ be a multisegment.
We let 
\[
E''_{\a}=\{(T', W'): T'\in Y_{\a}, W'\in Gr(\ell, \ker(T'|_{V_{\varphi_{\a}, k}}))\} .
\]
Note that we have a canonical morphism 
\[
 \alpha': E''_{\a}\rightarrow Gr(\ell, \varphi_{e(\a)}(k), V_{\varphi_{\a}, k})
\]
sending $(T',W')$ to $(W', \ker(T'|_{V_{\varphi_{\a}, k}}))$.

\end{definition}

\begin{prop}\label{prop: 7.7.6}
The morphism $\alpha'$ is a fibration.   
\end{prop}

\begin{proof}
The morphism $\alpha'$ is equivariant under the action of $GL(V_{\varphi_{\a}, k})$.
The same proof as in \cite[Proposition 5.28]{Deng23} shows that 
the morphism $\alpha'$ is actually a $P_{(U_1, U_2)}$ bundle, where
$P_{(U_1, U_2)}$  is a subgroup of $GL(V_{\varphi_{\a}, k})$ which fixes the given 
element $(U_1, U_2)$.  
Now we take a Zariski neighborhood $\mathfrak{U}$ of 
$(U_1, U_2)$ over which we have the trivialization
\[
 \gamma: \alpha'^{-1}(\mathfrak{U})\simeq \alpha'^{-1}((U_1, U_2))\times \mathfrak{U}, 
\]
such an isomorphism comes from a section 
\begin{equation}\label{eqn-section-to-group}
 s: \mathfrak{U} \rightarrow GL(V_{\varphi_{\a}, k}),
 ~s((U_1, U_2))=Id,
\end{equation}
given by $\gamma((T, W'))=[(g^{-1}T, g^{-1}W'), \alpha'((T, W'))]$, where $g=s(\alpha'((T, W')))$.
We remark that the existence of the section $s$
is guaranteed by local triviality of  $GL(V_{\varphi_{\a}, k})\rightarrow GL(V_{\varphi_{\a}, k})/P_{(U_1, U_2)}$,
cf.  \cite{S}, $\S$ 4.  
\end{proof}

Note that we use $(\mathfrak{X}_{\d})_W$  to denote the intersection of $\mathfrak{X}_{\d}$ with the fiber above $W$ along the morphism $\alpha$ discussed before Lemma \ref{lem: 7.7.2}.
\begin{prop}\label{prop: 7.7.7}
Assume that $\a$ is 
a multisegment satisfying $(\A_k)$
and $\d=\a+r[k+1]$ for some $r\leq \varphi_{e(\a)}(k)$.
Let $\ell\in \N$ such that $r+\ell=\varphi_{e(\a)}(k)$
and $W$ a subspace of 
$V_{\varphi_{\d}, k+1}$ 
such that $\dim(W)=r$.
We have a canonical projection
 \begin{equation}\label{eqn-proj-under-condition-A}
  p: (\mathfrak{X}_{\d})_W\rightarrow E''_{\a}
 \end{equation}
where for $T\in (\mathfrak{X}_{\d})_W$ with 
$\tau_W(T)=(T_1, T_0)\in (Z^{k+1, \d})_W\times \Hom(V_{\varphi_{\d}, k}, W)$, we define
$p(T)=(T_1, \ker(T_0|_{W_1}))$,  
where $W_1=\ker(T_1|_{V_{\varphi_{\d}, k}})$
(Note that here we identify $(Z^{k+1, \d})_W$ with $Y_{\a}$, see the remark after \cite[Proposition 5.31]{Deng23} ). 
Moreover, 
let $U_1\subseteq U_2\subseteq V_{\varphi_{\d}, k}$ be subspaces
such that $\dim(U_1)=\ell,~ \dim(U_2)=\varphi_{e(\a)}(k)$, then 
$p$ is a fibration 
with fiber 
$$\{T\in \Hom(V_{\varphi_{\d}, k}, W): \ker(T|_{U_2})=U_1 \}.$$
\end{prop}

\begin{proof}
We show that $p$ is well defined.
Since by definition of $(\mathfrak{X}_{\d})_W$,
\[
 \dim(W)+\dim(\ker(T_0|_{\ker(T_1|_{V_{\varphi_{\d}, k}})}))=\dim(\ker(T_1|_{V_{\varphi_{\d}, k}})),
\]
hence in order  to see that 
\[
 \ell =\dim(\ker(T_0|_{\ker(T_1|_{V_{\varphi_{\d}, k}})})),
\]
it suffices to show that 
\[
 \varphi_{e(\a)}(k)=\dim(\ker(T_1|_{V_{\varphi_{\d}, k}})).
\]
This follows from the fact that $\a=\d^{(k+1)}$. 
Finally, we show that $p$
is a fibration. 
Note that by definition, the fiber of $p$
is isomorphic to $$\{T\in \Hom(V_{\varphi_{\d}, k}, W): \ker(T|_{U_2})=U_1 \}.$$

So it suffices to show that it is locally trivial.
To show this, we consider the open subset $\mathfrak{U}$ in $E''_{\a}$ as constructed 
in the proof of Proposition \ref{prop: 7.7.6}.
Let $s$ be the section in (\ref{eqn-section-to-group}). We construct a trivialization 
for $p$:
\[
 \varrho: p^{-1}(\alpha'^{-1}(\mathfrak{U}))\rightarrow \alpha'^{-1}(\mathfrak{U})\times 
 \{T\in \Hom(V_{\varphi_{\d}, k}, W): \ker(T|_{U_2})=U_1 \}
\]
with $\varrho(T)=[(T_1, W'), g^{-1}(T_0)]$,
where $g=s((W', W_1))$, $W_1=\ker(T_1|_{V_{\varphi_{\d}, k}})$.
Note that given 
$$[(T_1, W'), T_0]\in \alpha'^{-1}(\mathfrak{U})\times 
 \{T\in \Hom(V_{\varphi_{\d}, k}, W): \ker(T|_{U_2})=U_1 \},$$ 
then $(W', W_1)\in \mathfrak{U}$ with $W_1=\ker(T_1|_{V_{\varphi_{\d}, k}})$.
It follows that we can take the section along $s$ in (\ref{eqn-section-to-group}) to obtain $g=s((W', W_1))$.
Let $T_0'=gT_0$. Then $T=\tau_W^{-1}((T_1, T_0'))\in  p^{-1}(\alpha'^{-1}(\mathfrak{U}))$.

\end{proof}

\begin{definition}
Let $\ell+\dim(W)=\varphi_{e(\a)}(k)$.
We define $ \mathfrak{Y}_{\a}$ to be the set of pairs $(T, U)$ satisfying
 \begin{description}
  \item[(1)] $U\in Gr(\ell, V_{\varphi_{\a},k})$,
  $T\in \End(V_{\varphi_{\a}}/U) \text{ of degree } 1$;
  \item[(2)] $ T\in O_{\b}$ for some $\b\preceq_k \a$ (cf. Definition \ref{def-new-partial-order}).
 \end{description}
 And we have a canonical 
 projection 
 \[
  \sigma: E''_{\a}\rightarrow  \mathfrak{Y}_{\a}
 \]
for $(T, U)\in  E''_{\a}$, we associate
\[
 \sigma((T, U))=(T', U)
\]
where 
 $T'\in \End(V_{\varphi_{\a}}/U)$
is the quotient of $T$. 
Also, we have a morphism 
\begin{equation}\label{def-eqn-sigma-p}
 \sigma': \mathfrak{Y}_{\a}\rightarrow Gr(\ell, \varphi_{e(\a)}(k), V_{\varphi_{\a}, k}),
\end{equation}
by 
\[
 \sigma'((T, U))=(U, \pi^{-1}(\ker(T|_{V_{\varphi_{\a},k}/U}))).
\]
where $\pi: V_{\varphi_{\a}, k}\rightarrow V_{\varphi_{\a},k}/U$
is the canonical projection.
\end{definition}

\begin{lemma}
 We have for $(T, U)\in \mathfrak{Y}_{\a}$,
\begin{description}
\item [(1)] $\gamma_k(T)\in Z^{k,\a}$ (cf. \cite[Definition 5.25]{Deng23}),
\item[(2)] $T|_{V_{\varphi_{\a}, k-1}} \text{ is surjective }$,
\item[(3)] $\dim(\ker(T|_{V_{\varphi_{\a},k}/U}))=\dim(W)$,
 \end{description}
where the morphism $\gamma_k: X_{k}^{\a}\rightarrow Z^{k,\a}$ is defined in \cite[After Definition 5.26]{Deng23}.
\end{lemma}

\begin{proof}
(1)To show $\gamma_k(T)\in Z^{k, \a}$, it suffices to show that 
 for $\b\preceq_k \a$ implies that $\b^{(k)}\leq \a^{(k)}$. 
 Note that by Proposition \ref{prop: 7.3.5},  
 there exists $\c\in S(\a)$ and $\Gamma\subseteq \c(k)$,
such that 
 \[
  \b=\c_{\Gamma}.
 \]
Now by Lemma \ref{lem: 7.6.2}, 
we have $\c\in \tilde{S}(\a)_k$, which
implies that 
\[
 \b^{(k)}=\c^{(k)}\leq \a^{(k)}.
\]
(2)By definition,
for any $(T, U)\in \mathfrak{Y}_{\a}$ we have $T\in O_{\b}$
for some $\b\preceq_k \a$. 
By the fact that   
$\a$ satisfies the assumption $(\A_k)$, 
any $\c\in S(\a)$
also satisfies $(\A_k)$, hence $\b=\c_{\Gamma}$
cannot contain a segment which starts at 
$k$, therefore $T|_{V_{\varphi_{\a}, k-1}}$ is surjective. \\
(3) 
Note that by the definition of $\mathfrak{Y}$, 
for  $(T, U)\in \mathfrak{Y}_{\a}$, we have 
$T\in O_{\b}$ for some $\b\preceq_k \a$. 
Now it follows 
\[
 \ker(T|_{V_{\varphi_{\a}, k}/U })=\varphi_{e(\a)}(k)-\ell=\dim(W).
\]

\end{proof}

\begin{prop}
Let $\a$ be a 
multisegment satisfying 
the assumption $(\A_k)$.
Then the morphism $\sigma'$ is a fibration. 
Moreover, if we assume that $\a=\a_{\Id}$
(cf. Lemma \ref{lem: 7.6.1}),
 then the morphism $\sigma$ is also a
fibration.  
\end{prop}

\begin{proof}
We first show that 
$\sigma'$ is locally trivial.
We observe that the group 
$GL(V_{\varphi_{\a}, k})$ acts both
on the source and target of $\sigma'$ 
in such a way that $\sigma'$ is $GL(V_{\varphi_{\a}, k})$-equivariant.
As in the proof of  
Proposition \ref{prop: 7.7.6}, 
let $\mathfrak{U}\subseteq Gr(\ell, \varphi_{e(\a)}(k), V_{\varphi_{\a}, k})$
be a neighborhood of a given element $(U_1, U_2)$
 such that we have a section 
\[
 s: \mathfrak{U}\rightarrow GL(V_{\varphi_{\a}, k}),~ s((U_1, U_2))=Id.
\]
 Note that in this case we have a natural trivialization of 
 $\sigma'$ by
 \[
  \sigma': \beta'^{-1}(\mathfrak{U})\simeq \mathfrak{U}\times \beta'^{-1}((U_1, U_2))
 \]
by $\sigma'((T, U))=[(U,  \pi^{-1}(\ker(T|_{V_{\varphi_{\a}, k}}))), g^{-1}((T, U))]$
with $g=s((U,  \pi^{-1}(\ker(T|_{V_{\varphi_{\a}, k}}))))$. 
Finally, we show that $\sigma$ 
 is surjective and locally trivial. 

 We observe that 
$\alpha'=\sigma'\sigma$ and $\sigma$ preserves fibers. 
 Now we fix a neighborhood $\mathfrak{U}$ as above and 
 get a commutative diagram
 \begin{displaymath}
  \xymatrix{
  \alpha'^{-1}(\mathfrak{U})\ar[d]\ar[r]^{\hspace{-1cm}\gamma }& \mathfrak{U}\times \alpha'^{-1}((U_1, U_2))\ar[d]^{\delta}\\
  \sigma'^{-1}(\mathfrak{U})\ar[r]^{\hspace{-1cm}\gamma'}& \mathfrak{U}\times \sigma'^{-1}((U_1, U_2))
  }
 \end{displaymath}
where $\delta([x, T])=[x, \sigma(T)]$ for any $x\in \mathfrak{U}$
and $T\in \alpha'^{-1}((U_1, U_2))$. Therefore to show that 
$\sigma $ is locally trivial , it suffices to show that it is locally trivial  
when restricted to the fiber $\alpha'^{-1}((U_1, U_2))$.
Note that we have 
\[
 \alpha'^{-1}((U_1, U_2))\simeq\{T\in Y_{\a}: \ker(T|_{V_{\varphi_{\a}, k}})=U_2\}\simeq (X_{\a}^{k})_{U_2}
 \hookrightarrow Y_{\a^{(k)}}\times \Hom(V_{\varphi_{\a}, k-1}, U_2)
\]
and 
\begin{align*}
\sigma'^{-1}((U_1, U_2))&\simeq \{T: T\in \End(V_{\varphi_{\a}}/U_1) \text{ of degree }1, 
\ker(T|_{V_{\varphi_{\a}, k}/U_1})=U_{2}/U_1, \\
   & T\in O_{\b}, \text{ for some } \b\preceq_k \a\} \hookrightarrow Y_{\a^{(k)}}\times \Hom(V_{\varphi_{\a}, k-1}, U_2/U_1).
\end{align*}
Note that the canonical morphism  
\[
 \Hom(V_{\varphi_{\a}, k-1}, U_2)\rightarrow \Hom(V_{\varphi_{\a}, k-1}, U_2/U_1)
\]
is a fibration. Hence to show that 
\[
 \alpha'^{-1}((U_1, U_2))\rightarrow \sigma'^{-1}((U_1, U_2))
\]
is a fibration, it suffices to show that $\sigma|_{ \alpha'^{-1}((U_1, U_2))}$ is surjective with isomorphic fibers everywhere . 
Let $(T, U_1)\in \sigma'^{-1}((U_1, U_2))$ with 
$$\tau_{U_2/U_1}(T)=(T_0, q_0)\in Y_{\a^{(k)}}\times \Hom(V_{\varphi_{\a}, k-1}, U_2/U_1),$$ 
where $\tau_{U_2/U_1}$ is the morphism defined in \cite[After Definition 5.26]{Deng23}(see also the remark before \cite[Definition 5.26]{Deng23}).
We fix a splitting $U_2\simeq U_2/U_1\oplus U_1$. Now to give $(T', U_1)\in \sigma^{-1}((T, U_1))$
amounts to give $q_1\in \Hom(V_{\varphi_{\a}, k-1}, U_1)$ such that 
\[
 \tau_{U_2}(T')=(T_0, q_0\oplus q_1).
\]
Note that by Lemma \ref{lem: 7.6.3},
the condition $\a=\a_{\Id}$ implies that 
$T'$ lies in $(X_{\a}^{k})_{U_2}$ if and only if 
$q_1$ satisfies 
\begin{align*}
 \dim(\ker(q_0\oplus q_1|_{\ker(T_0|_{V_{\varphi_{\a}, k-1}})}))=\varphi_{e(\a)}(k-1),
\end{align*}
which is an open condition. Therefore $\sigma$
is surjective. By definition of $\mathfrak{Y}_{\a}$,
\[
 \dim(\ker(q_0|_{\ker(T_0|_{V_{\varphi_{\a}, k-1}})}))=\varphi_{e(\a)}(k-1)+\ell,
\]
therefore if we denote $W_1=\ker(q_0|_{\ker(T_0|_{V_{\varphi_{\a}, k-1}})})$, 
then $q_1$ satisfies that 
\[
 \dim(\ker(q_1|_{\ker(T_0|_{V_{\varphi_{\a}, k-1}})})\cap W_1)=\varphi_{e(\a)}(k-1).
\]
Such a condition is independent of the 
pair $(T_0, q_0)$ since we always have 
$\dim(\ker(T_0|_{V_{\varphi_{\a}, k-1}}))=\varphi_{e(\a)}(k-1)+\varphi_{e(\a)}(k)$
and $\dim(W_1)=\varphi_{e(\a)}(k-1)+\ell$.
\end{proof}

We return to the morphism $p$ and $\sigma$.

\begin{lemma}\label{lem: 7.7.10}
Note that an element of $G_{\varphi_d}$ stabilizes $(\mathfrak{X}_{\d})_W$
if and only if it stabilizes $W$. Let $G_{\varphi_{\d}, W}$ be the 
stabilizer of $W$, then for $\c\leq \d$, and $T\in O_{\c}\cap (\mathfrak{X}_{\d})_W$, we have 
\[
 O_{\c}\cap  (\mathfrak{X}_{\d})_W=G_{\varphi_{\d}, W}T.
\]
\end{lemma}
\begin{proof}
 Recall that by \cite[Proposition 2.9]{Deng24b}, we have 
 \begin{displaymath}
  \xymatrix{
 X_{\d}^{ k+1}\ar[d] & GL_{\varphi_{\d}(k+1)}\times_{P_{W}}\alpha^{-1}(W)\ar[l]_{\hspace{-1.5cm}\delta}\ar[dl]\\
  Gr(\ell_{k+1}, V_{\varphi_{\d}})&
  }
\end{displaymath}
 where $\ell_{k+1}=\varphi_{e(\d)}(k+1)$. Note that we have 
 \[
  G_{\varphi_{\d}, W}=\cdots\times G_{\varphi_{\d},k}\times P_{W}\times G_{\varphi_{\d},k+2}\times \cdots ,
 \]
where $G_{\varphi_{\d}, i}=GL(V_{\varphi_{\d}, i})$. 
From this diagram we observe that  there is a one to one 
correspondance between the $G_{\varphi_{\d}}$ orbits 
on $X_{\d}^{k+1}$ and $G_{\varphi_{\d}, W}$ orbits on $\alpha^{-1}(W)$.  
Finally, since $\mathfrak{X}_{\d}$ is an open subvariety 
consisting of $G_{\varphi_{\d}}$ orbits, we are done.
\end{proof}

\begin{definition}
The canonical projection 
\[
 \pi: V_{\varphi_{\d}}\rightarrow V_{\varphi_{\d}}/W
\]
induces a projection 
\[
 \pi_{*}: G_{\varphi_{\d}, W}\rightarrow G_{\varphi_{\a}}, 
\]
where we identify $V_{\varphi_{\d}}/W$ with $V_{\varphi_{\a}}$.
\end{definition}

\begin{prop}
Assume that $\a$ is 
a multisegment satisfying $(\A_k)$
and $\d=\a+r[k+1]$ for some $r\leq \varphi_{e(\a)}(k)$.
The morphism $p$ is equivariant under the action of 
$G_{\varphi_{\d}, W}$ and $G_{\varphi_{\a}}$ via $\pi_{*}$, i.e, 
\[
 p(gx)=\pi_{*}(g)p(x).
\]
Moreover, it induces a one to one correspondance between orbits.
\end{prop}

\begin{proof}
Note that for 
$T\in (\mathfrak{X}_{\d})_W$, such that 
$\tau_W(T)=(T_1, T_0)\in (Z^{k+1, \d})_W\times \Hom(V_{\varphi_{\d}, k}, W)$, 
let
\[
U_1=\ker(T_1|_{V_{\varphi_{\d}, k}}),~ U_0=\ker(T_0|_{U_1})   
\]
we have 
\[
 p(T)=(T_1,U_0).
\]
Now it follows from the definition that 
we have 
\[
 p(gT)=\pi_{*}(g)p(T).
\]
Hence $p$ sends orbits to orbits. 
It remains to show that the pre-image of an orbit
is an orbit instead of a union of orbits.

We proved in Proposition \ref{prop: 7.7.7} that 
\[
 p^{-1}p(T)= \{(T_1, q): q\in \Hom(V_{\varphi_{\d}, k}, W), \ker(q|_{U_1})=U_0\},
\]
 note that here we identify elements of $(\mathfrak{X}_{\d})_{W}$ with 
 its image under $\tau_W$.
 Let $(T_1, q)\in p^{-1}p(T)$. Then 
 we want to find $g\in G_{\varphi_{\d}, W}$ such that $g(T_1, T_0)=(T_1, q)$. 
 Note that 
 by fixing a splitting $V_{\varphi_{\d},k+1}=W\oplus V_{\varphi_{\d},k+1}/W$,
 we can choose $g\in G_{\varphi_{\d}}$ such that 
 $g_{i}=Id\in GL(V_{\varphi_{\d}, i})$ for all $i\neq k+1$, and 
 \begin{displaymath}
  g_{k+1}=\begin{pmatrix}
     g_1 & g_{12}\\
     0& Id_{V_{\varphi_{\d},k+1}/W}
    \end{pmatrix}\in P_W,
 \end{displaymath}
where $g_1\in GL(W)$,  
and $g_{12}\in \Hom(V_{\varphi_{\d}, k+1}/W, W)$. 
By hypothesis,  the restrictions of $q$ 
and $T_0$ to 
$U_1$ are surjective with kernel $U_0$. So we 
can choose $g_1\in GL(W)$, such that 
\[
 g_1T_0(v)=q(v), \text{ for all } v\in U_1.
\]
Finally for $v_1\in V_{\varphi_{\d},k+1}/W$,  our assumption 
on $\a$ implies that 
$T_1|V_{\varphi_{\d}, k}$ is surjective. Hence there exists $v\in V_{\varphi_{\d}, k}$
such that $T_1(v)=v_1$.
Then we define  
 \[
  g_{12}(v_1)=q(v)-g_1T_0(v).
 \]
We check that this is well defined, i.e, for another $v'\in V_{\varphi_{\d}, k}$ such 
that $T_1(v')=v_1$, we have 
 \[
  q(v)-g_1T_0(v)=q(v')-g_1T_0(v'),
 \]
this is the same as to say that 
 \[
  q(v-v')=g_1T_0(v-v').
 \]
We observe that $T_1(v-v')=0$, hence $v-v'\in U_1$, now $ q(v-v')=g_1T_0(v-v')$
follows from our definition of $g_1$. Under such a choice, we have 
 \[
  g((T_1, T_0))=(T_0, q).
 \]
Hence we are done.
\end{proof}

\begin{prop}\label{prop: 7.6.17}
The morphism  $\sigma$ is equivariant under 
the action of $G_{\varphi_{\a}}$. 
Assume that $\a$ is a multisegment 
which satisfies the assumption $(\A_k)$. 
Let $\varphi\in \mathcal{S}$ such that 
\[
 \varphi+\l\chi_{[k]}=\varphi_{\a},
\]
where $\chi$ is the characteristic function at $k$.
Then there exists a one to one correspondance 
between the $G_{\varphi_{\a}}$-orbits on $\mathfrak{Y}_{\a}$ and 
the set 
\[
 S:=\{\b\in S(\varphi): \b\preceq_k \a\}.
\]
Moreover, for each orbit $\mathfrak{Y}_{\a}(\b)$
indexed by $\b$ on  $\mathfrak{Y}_{\a}$, $\sigma^{-1}(\mathfrak{Y}_{\a}(\b))$
is irreducible hence 
contains a unique orbit in $E''_{\a}$ as (Zariski) open subset.

\end{prop}

\begin{proof}
 The fact that 
 $\sigma$ is equivariant under the 
 action of $G_{\varphi_{\a}}$ follows 
 directly from the definition.
 To show that the $G_{\varphi_{\a}}$-orbits on
 $\mathfrak{Y}_{\a}$  
 is indexed by $S$, consider the morphism 
 \[
  p': \mathfrak{Y}_{\a}\rightarrow Gr(\ell, V_{\varphi_{\a}, k}), ~(T, U)\mapsto U
 \]
By \cite[Proposition 2.9]{Deng24b}, we have 
 the following diagram
 \begin{displaymath}
  \xymatrix{
\mathfrak{Y}_{\a} \ar[d]^{p'}& GL_{\varphi_{\a}(k)}\times_{P_{U}} p'^{-1}(U)\ar[l]_{\hspace{-1.5cm}\delta}\ar[dl]\\
  Gr(\ell, V_{\varphi_{\a}, k})&
  }
\end{displaymath}
 such that $p'$ is a $GL_{\varphi_{\a}, k}$ bundle. 
 Moreover, the same proof as in Lemma \ref{lem: 7.7.10} shows 
 that the orbits on $\mathfrak{Y}_{\a}$ are 
 in one to one correspondance with that of 
 the fibers 
 \begin{align*}
  p'^{-1}(U)\simeq \{T\in \End(V_{\varphi_{\a}}/U): &T \text{ is of degree }1, 
  T\in O_{\b} \text{ for some } \b\preceq_k \a
  \},
 \end{align*}
under the action of stabilizer $G_{\varphi_{\a}, U}$ of $U$. 
 Let $\varphi\in \mathcal{S}$ be the 
 such that $\varphi+\ell \chi_{[k]}=\varphi_{\a}$. 
 Then by identifying $V_{\varphi}$ with 
 $V_{\varphi_{\a}}/U$, we can view 
 $p'^{-1}(U)$ as an open subvariety of $E_{\varphi}$. 
 Note that we are identifying orbits with orbits by the canonical projection 
 \[
  G_{\varphi_{\a}, U}\rightarrow G_{\varphi}. 
 \]
Now it follows that 
the fibers are parametrized by the set $S$. 
Finally, let $\b\in S$. We have to show that 
$\sigma^{-1}(\mathfrak{Y}_{\a}(\b))$ is irreducible, which
is a consequence of the following lemma.

\end{proof}

\begin{lemma}\label{lem: 7.6.18}
Let $\a, \b$ be the  multisegments as above. Then 
there exists a 
bijection between the set  
\[
 Q(\a, \b)=\{\c\in S(\a): \b=\c_{\Gamma} \text{ for some } \Gamma\subseteq \c(k)\},
\]
and the orbits in $\sigma^{-1}(\mathfrak{Y}_{\a}(\b))$
which respects the poset structure  given by 
\[
 \c\mapsto E''_{\a}(\c^{\sharp}),
\]
where for $\b=\c_{\Gamma}$,
\[
  \c^{\sharp}=(\c\setminus \c(k))\cup \Gamma\cup \{\Delta^+: \Delta\in \c(k)\setminus \Gamma\},
\]
and $E''_{\a}(\c^{\sharp})$ is the orbit indexed by $\c^{\sharp}$ on $E''_{\a}$.
Moreover, the set $Q(\a, \b)$
contains a unique minimal element.
\end{lemma}

\begin{proof}
We constructed in (\ref{eqn-proj-under-condition-A}) of Proposition 
\ref{prop: 7.7.7} a morphism $p$. Consider the composition
\[
 (\mathfrak{X}_{\d})_W \xrightarrow{p} E_{\a}''\xrightarrow{\sigma} \mathfrak{Y}_{\a},
\]
which sends $(O_{\c})_W$ to $\mathfrak{Y}(\b)$, where 
$\b=\c^{(k, k+1)}$ for $\c \in S(\d)$. Hence we have 
\[
 \b=(\c^{(k+1)})_{\Gamma}
\]
for $\Gamma=\{\Delta\in \c: e(\Delta)=k\}$. Note that 
$\c\in S(\d)$ implies that  
 $\c^{(k+1)}\leq \a=\d^{(k+1)}$.
 Conversely, for $\c\in Q(\a, \b)$, 
 such that 
 \[
  \b=\c_{\Gamma},
 \]
there is a unique element 
\[
 \c'=\c^{\sharp}
\]
in $S(\d)$ such that $O_{\c'}\subseteq \mathfrak{X}_{\d}$ and $\c=\c'^{(k+1)}$.
Therefore we conclude that 
there is a bijection between 
the $G_{\varphi_{\a}}$-orbits in $\sigma^{-1}(\mathfrak{Y}_{\a}(\b))$
and $Q(\a, \b)$.  

Finally, 
for 
\[
 \varphi_{\a}=\varphi_{\b}+\ell\chi_{[k]},
\]
we show
by induction on $\ell$
that the set 
$Q(\a, \b)$
contains a unique minimal element.

For case 
$\ell=1$, let 
\[
 \b(k): =\{\Delta\in \b: e(\Delta)=k\}=\{\Delta_1\preceq \cdots \preceq \Delta_h\}.
\]
and $\c_i=(\b\setminus \Delta_i)\cup \Delta_i^+$.
Then 
\[
 Q(\a, \b)\subseteq \{\c_i: i=1, \cdots, h\},
\]
and $\c_h$ is minimal in the latter, which 
implies that $\c_h\in Q(\a, \b)$ and is minimal. 
In general,
let 
\[
 \varphi=\varphi_{\b}+\chi_{[k]}.
\]
Note that there exists 
$\c'\in S(\varphi) $
satisfying the assumption $(\A_k)$
and $\Gamma'\subseteq \c'(k)$ 
such that 
\[
 \b=\c'_{\Gamma'}. 
\]
In fact, by assumption
\[
 \b=\c_{\Gamma}
\]
for some $\c\in S(\a)$ and $\Gamma\subseteq \c(k)$. 
Let 
\[
 \Gamma\supseteq \Gamma_1,
\]
such that $\ell=\sharp \Gamma=\sharp \Gamma'+1$ and
\[
 \c'=\c_{\Gamma_1},   
\]
then we have 
\[
 \b=\c'_{\Gamma\setminus \Gamma_1}.
\]
Now we apply our induction to 
the case 
 \[
  Q_1: =\{\c\in S(\varphi): 
  \c \text{ satisfies the assumption } (\A_{k}),
  \b=\c_{\Gamma} \text{ for some } \Gamma\subseteq \c(k)\},
 \]
from which we know that there exists a unique minimal element $\c_1$ in $ Q_1$. 
Now by assumption
\[
 \b_1\leq \c'\preceq_k \a,
\]
and by induction the set 
$Q(\a, \b_1)$ contains a unique element $\b_2$.
We claim that $\b_2$ is minimal in $Q(\a, \b)$. 
In fact, let $\e\in Q(\a, \b)$, then 
\[
 \b=\e_{\Gamma'}
\]
for some $\Gamma'\subseteq \e(k)$. 
Again let 
\[
 \Gamma_1'\subseteq \Gamma',~ \e'=\e_{\Gamma_1'}
\]
such that $\ell=\sharp \Gamma'=\sharp \Gamma_1'+1$.
Now we obtain 
\[
 \e'\in Q_1, ~ \b=\e'_{\Gamma'\setminus \Gamma_1'}.
\]
By minimality of $\c_1$, we know that 
\[
 \c_1\leq \e'.
\]
Note that this implies $\c_1\preceq \e'$, and 
by transitivity of poset relation, we get 
$\c_1\preceq_k \e$. Now we apply 
Proposition \ref{prop: 7.3.5} to get 
\[
 \c_1=\f_{\Gamma''}, 
\]
for some $\f\in S(\e)$ and $\Gamma''\subseteq \f(k)$.
Again we deduce from induction that 
\[
 \f\geq \c_2.
\]
Hence $\c_2\leq \e$. 
\end{proof}

Now we return to the calculation
of product of perverse sheaves, cf. 
Corollary \ref{cor: 7.4.15}.

\begin{teo}\label{cor: 7.4.19}
Let $\a$ be a multisegment 
satisfying the assumption $(\A_k)$
and $\b\preceq_k \a$ such that  
\[
 \varphi_{\a}=\varphi_{\b}+\ell\chi_{[k]}, \quad \ell\in \N.
\]
Let $\c$ the minimal element 
in $Q(\a, \b)$ and $E''_{\a}(\c)$ be the 
$G_{\varphi_{\a}}$ orbit indexed by $\c$
in $E_{\a}''$ (cf. Definition \ref{def-open-subset-Ea}).  
Then we have 
\[
 IC(\line{O}_\b)\star IC(\line{O}_{\l [k]})=\beta''_{*}(IC(\line{E_{\a}''(\c^{\sharp})})). 
\]
\end{teo}
 
\begin{proof}

First of all, by
definition 
\[
 E''=\{(T, U): T\in E_{\varphi_{\a}}, ~T(U)=0, ~ \dim(U)=\ell \},
\]
therefore we have 
\[
 E''_{\a}\subseteq E''.
\]
Furthermore, the variety $ E''_{\a}$ is open in $E''$. 
In fact, consider the canonical
morphism 
\[
 \beta'': E''\rightarrow E_{\varphi_{\a}},
\]
then $E''_{\a}=\beta''^{-1}(Y_{\a})$. 
Since $Y_{\a}$ is open in $E_{\varphi_{\a}}$,
 $ E''_{\a}$ is open in $E''$.
Now we have two morphisms 
\begin{align*}
\sigma\beta': & \beta'^{-1}(E''_{\a})\rightarrow \mathfrak{Y}_{\a},\\
\beta: & E'\rightarrow E_{\varphi_{\b}}\times E_{\varphi_{\ell[k]}}\simeq E_{\varphi_{\b}}.
\end{align*}

We claim that $\beta^{-1}(O_{\b})\cap  \beta'^{-1}(E''_{\a})=\beta'^{-1}\sigma^{-1}(\mathfrak{Y}_{\a}(\b))$,
where $\mathfrak{Y}_{\a}(\b)$ is the orbit in $\mathfrak{Y}_{\a}(\b)$ under the 
action of $G_{\varphi_{\a}}$. 

By definition of $\beta$,
\begin{align*}
 &\beta^{-1}(O_{\b})\cap \beta'^{-1}(E''_{\a})&=\{(T, W, \mu, \mu')| \mu: W\simeq V_{\varphi_{\ell[k]}}, ~\mu': V_{\varphi_{\a}}/W\simeq V_{\varphi_{\b}},\\
& &T\in O_{\f} \text{ for some } \f\in S(\a), \b\preceq_k \f\}.
\end{align*}
Now by definition of $\sigma$ and $\beta'$, 
we know that 
$\beta^{-1}(O_{\b})\cap  \beta'^{-1}(E''_{\a})=\beta'^{-1}\sigma^{-1}(\mathfrak{Y}_{\a}(\b))$.
Now by Proposition \ref{prop: 7.6.17},
$\sigma^{-1}(\mathfrak{Y}_{\a}(\b))$ 
contains $E_{\a}''(\c^{\sharp})$ as the unique open 
orbit, where $\c$ is the minimal
element in $Q(\a, \b)$.
Therefore we conclude that 
\[
  \beta'^{*}(IC(\line{E''_{\a}(\c^{\sharp})}))=\beta^*(IC(\line{O_{\b}})\otimes IC(E_{\varphi_{\ell[k]}})).
\]
Now by definition
\[
  IC(\line{O}_\b)\star IC(\line{O}_{\l [k]})=\beta''_{*}(IC(\line{E_{\a}''(\c^{\sharp})})).
\]

\end{proof}

\section{Multisegments of Grassmanian Type}\label{sec-geom-grass}

In order to precisely describe the previous corollary
concerning Lusztig's product in the Grassmanian case 
in the next section, 
we  generalize the construction
in \cite[\S 5.3]{Deng23} to get more general results concerning the 
the set $S(\a)$ for general multisegment $\a$.

Let $V$ a $\C$ vector space of dimension $r+\l$  and $Gr_r(V)$ be the variety of 
$r$-dimensional subspaces of $V$. 

\begin{definition}
By a partition of $ \l$, we mean a sequence $\lambda=(\ell_1, \cdots, \ell_r)$ for some $r$, where $\ell_i\in \N$
, $0\leq \ell_1\leq \cdots \ell_r\leq \ell$.
And for $\mu=(\mu_1,\cdots, \mu_s)$ be another partition, we say $\mu\leq \lambda$
if and only if $\mu_i\leq \lambda_i$ for all $i=1, \cdots, $. Let $\mathcal{P}(\ell, r)$
be the set of partitions of $\ell$ into $r$ parts.

\end{definition}

\begin{definition}
We fix a complete flag 
\[
 0=V^0\subset V^1\subset \cdots \subset V^{r+\ell}=V.
\]
This flag provides us a stratification of the variety $Gr_r(V)$ by Schubert varieties, labeling  by partitions 
, denoted by $\line{X}_{\lambda}$, 
\[
\line{X}_{\lambda}=\{U\in Gr_r(V): \dim(U\cap V^{\ell_i+i})\geq i, \text{ for all } i=1, \cdots, r \}.
\]
\end{definition}

\begin{lemma}(cf. \cite{Z5})
We have 
\[
 \mu\leq \lambda \Longleftrightarrow \line{X}_{\mu}\subseteq \line{X}_{\lambda}.
\]
And the Schubert cell 
\[
 X_{\lambda}=\line{X}_{\lambda}-\sum_{\mu<\lambda}\line{X}_{\mu}
\]
is open in $\line{X}_{\lambda}$.
\end{lemma}

\begin{definition}
Let $\Omega^{r, \ell}$ be the set 
\begin{align*}
 \Omega^{r, \l}=\{(a_1, \cdots, a_m; b_0, \cdots, b_{m-1}): &\sum_i a_i=r, \sum_j, b_j=\l,\\ 
 &\text{ for } 0<i<m, a_i>0, b_i>0\}.
\end{align*}
\end{definition}

\begin{lemma}(cf. \cite{Z5})\label{lem: 7.2.4}
There exists a bijection
\[
 \Omega^{r, \l}\rightarrow \mathcal{P}(\ell, r),
\]
which sends $(a_1, \cdots, a_m; b_0, \cdots, b_{m-1})$ to a partition of $\ell$
given by $~b_0, ~b_0+b_1, \cdots,~, b_0+\cdots+b_{m-1}$, and that the elements $b_0+\cdots+b_{i-1}$ figures in $\lambda$
with multiplicity $a_i$. 

\end{lemma}

\begin{notation}
 From now on, we will also write 
\[
 \lambda=(a_1, \cdots, a_m; b_0, \cdots, b_{m-1}),
\]
with notations as in the previous lemma.
\end{notation}

%

We introduce the formula in \cite{Z5} to calculate the 
Kazhdan Lusztig polynomials for Grassmannians. 

\begin{definition}
Let $\lambda=(a_1, \cdots, a_{m}; b_0, \cdots, b_{m-1})$ be a partition.
Following \cite{Z5}, we represent a partition as a broken line in the 
plane $(x, y)$, i.e, the graph of the piecewise-linear function $y=\lambda(x)$
which equals $|x|$ for large $|x|$, has everywhere slope $\pm 1$, and whose 
ascending and decreasing segments are precisely $ b_0, \cdots, b_{m-1}$ 
and $a_1, \cdots, a_m$, respectively. 
Moreover, we call the local maximum and minimum of the graph $y=\lambda(x)$ 
the peaks and depressions of $\lambda$.

\end{definition}

\begin{figure}[!ht]
\centering
\includegraphics{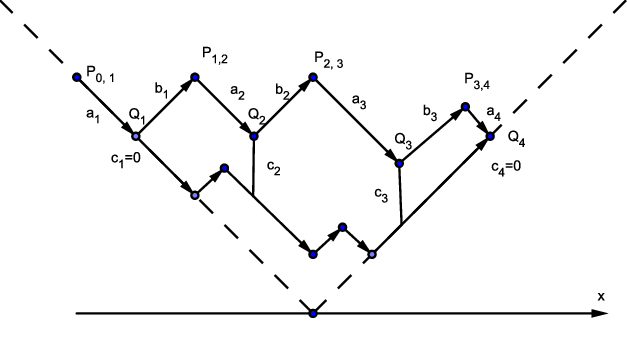}
\caption{\label{fig-multisegment15} }
\end{figure}

%
%
%

\begin{lemma}(cf. \cite{Z5})
For $\lambda, \mu\in \Omega^{r, \l}$, then  
\[
 \lambda\geq \mu \Longleftrightarrow \lambda(x)\geq \mu(x), \text{ for all }x.
\]
\end{lemma}

From now on until the end of this section, let 
\[
 J=\{\sigma_i: i=1, \cdots, r-1\}\cup \{\sigma_i: i=r+1\cdots, r+\ell-1\},
\]
and fix
\[
 \a: =\a_{\Id}^{J, \emptyset}=\{\Delta_1, \cdots, \Delta_r, \cdots, \Delta_{r+\l}\}
\]
to be a multisegment of parabolic type $(J, \emptyset)$ (cf. \cite[Definition 3.26]{Deng24b}) such that
\[
e(\Delta_i)=k-1, \text{ for } i=1, \cdots, r,
\]
and 
\[
 e(\Delta_i)=k, \text{ for } i=r+1, \cdots, r+\l. 
\]

\begin{definition} 
Then to each partition $\lambda\in \Omega^{r_1, \l_1}$ such that 
$r_1\geq r$ and $r_1+\l_1=r+\l$, we associate 
\begin{align*}
 \a_{\lambda}&=\sum_{i=1}^{b_0}[b(\Delta_{i}), k]+\sum_{i=b_0+1}^{b_0+a_1}[b(\Delta_{i}), k-1]
 +\cdots\\
 &+\sum_{i=b_0+a_1\cdots+b_{j-1}+1}^{b_0+a_1\cdots+b_{j-1}+a_{j}}[b(\Delta_{i}), k-1]+
 \sum_{i=b_0+\cdots +b_{j-1}+a_{j}+1}^{b_0+a_1\cdots+b_j}[b(\Delta_{i}), k]+\cdots. 
\end{align*} 
 
\end{definition}

\begin{definition}
Let $r, n\in \N$ such that $r\leq n$. Let 
\[
 R_r(n)=\{(x_1, \cdots, x_r): 1\leq x_1<\cdots< x_r\leq n \}.
\]

\begin{description}
 \item[(1)] 
 Let $x=(x_1, \cdots, x_{r_1})\in R_{r_1}(n)$ and $x'=(x_1', \cdots, x_{r_2}')\in R_{r_2}(n)$
such that $r_1\geq r_2$. We say
$x\supseteq x'$ if 
$\{x_1, \cdots, x_{r_1}\}\supseteq \{x_1', \cdots, x_{r_2}'\}$.
\item[(2)]Let $x=(x_1, \cdots, x_{r})\in R_{r}(n)$ and $x'=(x_1', \cdots, x_{r}')\in R_{r}(n)$.
We say $x\geq x'$ if $x_i\geq x_i'$ for all $i=1, \cdots, r$.
\item[(3)] We define $x\succeq y$, if $x\geq y'\supseteq y$ for some $y'$.
\end{description}
\end{definition}

\remk The set $R_r(n)$ is a poset with respect to the relation $\geq$, while the set $\cup_{r\leq n}R_r(n)$ is a poset with respect to the relation $\supseteq$.

\begin{prop}
For $J=\{\sigma_i: i=1, \cdots, r-1\}\cup \{\sigma_i: i=r+1\cdots, r+\l-1\}$,
we have an isomorphism of posets
$$\varsigma_1: S_{r+\l}^{J, \emptyset}\rightarrow R_r(r+\l),$$
by associating the element
$w$ with $x_w:=(w^{-1}(1), \cdots, w^{-1}(r))$.  
\end{prop}

\begin{proof}
Note that 
by definition (cf. \cite[Definition3.20]{Deng24b}), 
\[
S_{r+\l}^{J, \emptyset}=\{w\in S_{r+\l}: w^{-1}(1)<\cdots< w^{-1}(r) \text{ and } w^{-1}(r+1)<\cdots <w^{-1}(r+\l)\}.
\]
Therefore,
$\varsigma$ is a bijection. This preserves the partial order by \cite[Proposition 2.4.8]{BF}.
\end{proof}

\begin{definition}
For $\lambda\in \Omega^{r, \l}$ and $\lambda'\in \Omega^{r_1, \l_1}$ 
such that $r+\l=r_1+\l_1$.
We define 
$\lambda\supseteq \lambda'$ if and only if $x_{\lambda}\supseteq x_{\lambda'}$,
and $\lambda\succeq \lambda'$ if and only if $x_{\lambda}\succeq x_{\lambda'}$.

\end{definition}

\begin{definition}

Let $\lambda=(a_1, \cdots, a_m; b_0, \cdots, b_{m-1})$, 
consider the set 
$$\{b(\Delta): \Delta\in \a_{\lambda}, e(\Delta)=k-1\}=\{x_1<\cdots<x_r\},$$
here we have $r$ segments ending in $k-1$ since $\sum_ia_i=r$,
we associate $\lambda$ with the element 
\[
 x_{\lambda}: =(x_1,\cdots, x_r).
\]
This allows us to get a morphism $\varsigma_2: \Omega^{r, \l}\rightarrow R_r(r+\l)$
sending $\lambda$ to $x_{\lambda}$.
\end{definition}

\begin{lemma}
 The map $\varsigma_2$ is an isomorphism of posets. 
\end{lemma}
 
\begin{proof}
To see that $\varsigma_2$ is a bijection, we only need to construct an inverse.
Given $x=(x_1, \cdots, x_r)\in R_r(r+\l)$, we have 
$y=(y_1, \cdots, y_{\l})\in R_{\l}(r+\l)$  such that $\{1, \cdots, r+\l\}=\{x_1, \cdots, x_r, y_1, \cdots, y_{\l}\}$.
We can associate a multisegment to $x$:
\[
 \a_x=\sum_{j=1}^r[b(\Delta_{x_{j}}), k-1]+\sum_{j=1}^{\l}[b(\Delta_{y_j}), k].
\]
Note that this allows us to construct a 
partition $\lambda(x)\in \Omega^{r, \l}$ by counting the segments ending in 
$k-1$ and $k$ respectively. 

A simple calculation shows that 
if we write $\lambda\in \Omega^{r, \ell}$ is identified via Lemma \ref{lem: 7.2.4} with a partition $(\l_1, \cdots, \l_r)$ such that $0\leq \l_1\leq \cdots\leq \l_r$,
then 
\[
 \varsigma_2(\lambda)=(\l_1+1, \cdots, \l_r+r).
\]
It follows from \cite[Page 6, Paragraph 3]{Br} that 
\[
 \mu\geq \lambda\Leftrightarrow \varsigma_2(\mu)\geq \varsigma_2(\lambda).
\]
\end{proof}

\begin{prop}\label{prop: 7.3.3}
For $\lambda\in \Omega^{r, \l}$,
we have $\a_{\lambda}\in S(\a)$, moreover, all the elements in $S(\a)$ are of this form. 
Moreover, we have $S(\a_{\lambda})=\{a_{\mu}: \mu\geq \lambda\}$.
\end{prop}
 
\begin{proof}
Let $w\in S^{J, \emptyset}$, by definition, we have 
\[
 w^{-1}(1)<\cdots<w^{-1}(r), ~w^{-1}(r+1)<\cdots <w^{-1}(r+\l).
\]
By definition, we have 
\begin{align*}
\Phi_{J, \emptyset}(w)=&\sum_{j}[b(\Delta_j), e(\Delta_{w(j)})]\\
=&\sum_{j}[b(\Delta_{w^{-1}(j)}), e(\Delta_{j})]\\
=&\sum_{j=1}^{r}[b(\Delta_{w^{-1}(j)}), k-1]+\sum_{j=r+1}^{r+\l}[b(\Delta_{w^{-1}(j)}), k]\\
=\a_{\varsigma_2^{-1}(x_w)}
\end{align*}

Now that $\varsigma_2^{-1}\circ\varsigma_1$ preserves
the partial order, we have 
\[
 S(\a_{\lambda})=\{\a_{\mu}: \mu\geq \lambda\}
\]
by \cite[Proposition 3.18]{Deng24b}.

\end{proof}

\begin{example}

For example, for $r=1, \ell=3$, with 
$J=\{\sigma_2, \sigma_3\}$ and 
\[
 \a=\a_{\Id}^{J,\emptyset}=[1, 4]+[2, 5]+[3, 5]+[4, 5].
\]

Let $\lambda=(a_1, a_2; b_0, b_1)=(1, 0; 2, 1)$, then $\a_{\lambda}=[1, 5]+[2, 5]+[3, 4]+[4, 5]$.
This corresponds to the element $\varsigma_1^{-1}\circ \varsigma_2(\lambda)=\sigma_1\sigma_2$
in $S_4^{J, \emptyset}$.

\begin{figure}[!ht]
\centering
\includegraphics{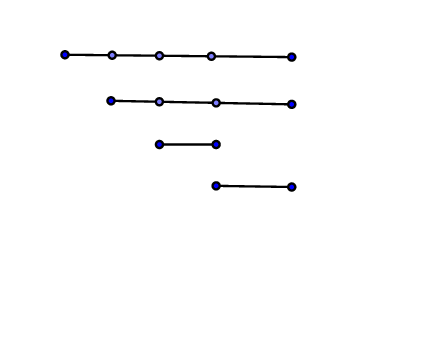}
\caption{\label{fig-multisegment14} }
\end{figure}

\end{example}

\begin{prop}
Let $\lambda, \mu\in \Omega^{r, \l}$ such that $\lambda<\mu$.
We have 
\[
 P_{\a_{\lambda}, \a_{\mu}}(q)=P_{\lambda, \mu}(q).
\] 
\end{prop}

\begin{proof}
 We can also prove this proposition in the following way. 
Let $w, v\in S_{r+\l}^{J, \emptyset}$, 
such that 
\[
 \lambda=\varsigma_2^{-1}\varsigma_1(w), ~\mu=\varsigma_2^{-1}\varsigma_1(v).
\]
Let $P_{J}$ be the parabolic
subgroup of $GL_n$, then by fixing an element in $V_0\in Gr_r(\C^{r+\l})$,  
we can identify $P_J\backslash GL_n$ with $Gr_r(\C^{r+\l})$. 
Moreover, the $B$-orbits $P_J\backslash wB$ corresponds to the varieties
$X_{\lambda}$, see \cite{Br} for a precise description.
Hence we have 
\[
 P_{\lambda, \mu}(q)=P^{J, \emptyset}_{w, v}(q)=P_{a_{\lambda}, \a_{\mu}}(q).
\]
\end{proof}

\begin{definition}
Let $\lambda\in \Omega^{r, \l}$. 
\begin{description}
 \item[(1)]
We define  
\[
 \Gamma(\lambda)=\{\mu\in \Omega^{r_1, \l_1}: r_1+\l_1=r+\l, r_1\geq r,~
\mu\succeq \lambda\}.
\]
and 
\[
 \Gamma^{\mu}(\lambda)=\{\mu': \mu\geq \mu', \mu'\succeq \lambda\},
\]
\[
  \Gamma_1^{\mu}(\lambda)=\{\mu': \mu\geq \mu', \mu'\supseteq \lambda\}.
\]

\item[(2)] For $\mu\in \Gamma(\lambda)$, we define 
\[
 S^{\mu}(\lambda)=\{\lambda'\in \Omega^{r, \l}: \lambda'\geq \lambda, \mu\succeq\lambda'\},
\]
and let 
\[
 S_1^{\mu}(\lambda)=\{\lambda'\in \Omega^{r, \l}: \lambda'\geq \lambda, \mu\supseteq \lambda'\}.
\]
\end{description}

\end{definition}

\begin{prop}\label{prop: 7.3.9}
Let $\lambda\in \Omega^{r, \l}$ and $\mu\in \Omega^{r_1, \l_1}$ with $r_1\geq r$
and $r_1+\l_1=r+\l$.  
Then $\pi(\a_{\mu})$ appears as a summand of $\D^k(\pi(\a_{\lambda}))$
if and only if $\mu\in \Gamma(\lambda)$.  
\end{prop}

\begin{proof}
Let $x_{\lambda}=(x_1^{\lambda}, \cdots, x_r^{\lambda})=\varsigma_2(\lambda)$
and $y_{\lambda}=(y_1^{\lambda}, \cdots, y_{\l}^{\lambda})\in R_r(r+\l)$ such that 
\[
 \{1, \cdots, r+\l\}=\{x_1^{\lambda}, \cdots, x_r^{\lambda}, y_1^{\lambda}, \cdots, y_{\l}^{\lambda}\}.
\]
By Proposition \ref{prop: 7.3.3}, we have 
\[
 \a_{\lambda}=\sum_{j=1}^r[b(\Delta_{x_{j}^{\lambda}}), k-1]+\sum_{j=1}^{\l}[b(\Delta_{y_j^{\lambda}}), k].
\]
Therefore
\[
 \D^k(\pi(\a_{\lambda}))=\pi(\a_{\lambda})+\sum_{y\supseteq x_{\lambda}}\pi(\a_{\varsigma_2^{-1}(y)}).
\]
By Lemma \ref{lem: 3.0.8},
$\pi(\mu)$ is a summand of 
$\D^k(\pi(\a_{\lambda}))$ if and only if 
$\mu\geq \varsigma_2^{-1}(y)$ for some $y\supseteq x_{\lambda}$, i.e,
$\mu\succeq \lambda$.
\end{proof}

\begin{cor}
We have $\mu\succeq \lambda$ if and 
only if $\a_{\mu}\preceq_k \a_{\lambda}$. 
 
\end{cor}

\begin{proof}
By Corollary \ref{cor: 7.4.3}, 
 $\a_{\mu}\preceq_k \a_{\lambda}$ if and only if $\D^k(\pi(\a_{\lambda}))-\pi(\a_{\mu})\geq 0$
 in $\mathcal{R}$, which is equivalent to say that $\mu\preceq \lambda$ by 
 the previous proposition. 
\end{proof}

\begin{prop}
Let $\lambda\in \Omega^{r, \ell}$ 
and $\mu\in \Omega^{r_1, \ell_1}$. Then we have  
$\a_{\mu}=(\a_{\lambda})_{\Gamma}$
 for some $\Gamma\subseteq \a_{\lambda}(k)$. 
if and only if we have $\mu\supseteq \lambda$. 
\end{prop}

\begin{proof}
Let $x_{\lambda}=(x_1^{\lambda}, \cdots, x_r^{\lambda})=\varsigma_2(\lambda)$
and $y_{\lambda}=(y_1^{\lambda}, \cdots, y_{\l}^{\lambda})\in R_r(r+\l)$ such that 
\[
 \{1, \cdots, r+\l\}=\{x_1^{\lambda}, \cdots, x_r^{\lambda}, y_1^{\lambda}, \cdots, y_{\l}^{\lambda}\}.
\]
By Proposition \ref{prop: 7.3.3}, we have 
\[
 \a_{\lambda}=\sum_{j=1}^r[b(\Delta_{x_{j}^{\lambda}}), k-1]+\sum_{j=1}^{\l}[b(\Delta_{y_j^{\lambda}}), k].
\] 
Hence
\[
 \a_{\lambda}(k)=\sum_{j=1}^{\l}[b(\Delta_{y_j^{\lambda}}), k].
\]
Let 
$\Gamma=\sum_{m=1}^{t}[b(\Delta_{y_{j_m}^{\lambda}}), k]$. 
If $\a_{\mu}=(\a_{\lambda})_{\Gamma}$, then 
\[
 \a_{\mu}=\sum_{j=1}^r[b(\Delta_{x_{j}^{\lambda}}), k-1]+\sum_{m=1}^{t}[b(\Delta_{y_{j_m}^{\lambda}}), k-1]
 +\sum_{ j\notin \{j_1, \cdots, j_t\}}[b(\Delta_{y_j^{\lambda}}), k].
\]
Therefore 
\[
 x_{\mu}\supseteq x_{\lambda}
\]
as a set.
The converse is also true. 
\end{proof}

\section{Grassmanian case}\label{section-Grass}

As before, let 
\[
 J=\{\sigma_i: i=1, \cdots, r-1\}\cup \{\sigma_i: i=r+1\cdots, r+\ell-1\},
\]
and
\[
 \a: =\a_{\Id}^{J, \emptyset}=\{\Delta_1, \cdots, \Delta_r, \cdots, \Delta_{r+\l}\}
\]
be a multisegment of parabolic type $(J, \emptyset)$  (cf. \cite[Definition 3.26]{Deng24b}) such that
\[
e(\Delta_i)=k-1, \text{ for } i=1, \cdots, r,
\]
and 
\[
 e(\Delta_i)=k, \text{ for } i=r+1, \cdots, r+\l. 
\]
Moreover, for $\lambda\in \mathcal{P}(\ell, r)$, let 
$x_{\lambda}=(x_1^{\lambda}, \cdots, x_r^{\lambda})=\varsigma_2(\lambda)\in R_r(r+\l)$
and $y_{\lambda}=(y_1^{\lambda}, \cdots, y_{\l}^{\lambda})\in R_\l(r+\l)$ such that 
\[
 \{1, \cdots, r+\l\}=\{x_1^{\lambda}, \cdots, x_r^{\lambda}, y_1^{\lambda}, \cdots, y_{\l}^{\lambda}\}.
\]
It follows from Proposition \ref{prop: 7.3.3} that we have
\[
 \a_{\lambda}=\sum_{j=1}^r[b(\Delta_{x_{j}^{\lambda}}), k-1]+\sum_{j=1}^{\l}[b(\Delta_{y_j^{\lambda}}), k].
\]    
Let $0<r_0\leq \ell $ and $r_1=r+r_0, ~ \ell_1=\ell-r_0$.

\begin{prop}\label{prop: 7.6.1}
Let $\mu\in \mathcal{P}(\ell_1, r_1)$. Then 
there exists $\mu^{\flat}\in \mathcal{P}(\ell, r)$, such that 
\[
 \{\b\in S(\a): \a_{\mu}\preceq_k \b\}=\{\a_{\lambda}: \lambda\in \mathcal{P}(\ell, r),~ \lambda\leq \mu^{\flat}\}.
\]
More explicitly, if 
$x_{\mu}=(x_1^{\mu}, \cdots, x_{r_1}^{\mu})=\varsigma_2(\mu)$, then 
\[
 x_{\mu^{\flat}}=\varsigma_2(\mu^{\flat})=(x_{r_0+1}^{\mu}, \cdots, x_{r_1}^{\mu}).
\]

\end{prop}
     
\begin{proof}
By Lemma \ref{lem: 7.6.18},  the set  
\[
  \{\b\in S(\a): \a_{\mu}\preceq_k \b\}
\]
contains a unique minimal element 
$\a_{\mu^{\flat}}\in S(\a)$ for some $\mu^{\flat}\in \mathcal{P}(\ell, r)$. 
Therefore we have 
\[
 \{\b\in S(\a): \a_{\mu}\preceq_k \b\}=\{\a_{\lambda}: \lambda\in \mathcal{P}(\ell, r),~ \lambda\leq \mu^{\flat}\}.
\]
Note that if we write 
\[
 \a_{\mu}=\sum_{j=1}^{r_1}[b(\Delta_{x_j^{\mu}}), k-1]+\sum_{j=1}^{\l_1}[b(\Delta_{y_j^{\mu}}), k],
\]
then 
\[
 \a_{\mu^{\flat}}=\sum_{j=1}^{r_0}[b(\Delta_{x_j^{\mu}}), k]+\sum_{j=r_0+1}^{r_1}[b(\Delta_{x_j^{\mu}}), k-1]+\sum_{j=1}^{\l_1}[b(\Delta_{y_j^{\mu}}), k]
\]
is the minimal element in $S(\a)$ satisfying
\[
 \a_{\mu}=(\a_{\mu^{\flat}})_{\Gamma}, \quad ( cf. \text{ Definition \ref{def-new-partial-order}})
\]
for some $\Gamma\subseteq \a_{\mu^{\flat}}(k)$. 
\end{proof}

\begin{definition} 

Let 
\[
 J_1=\{\sigma_i: i=1, \cdots, r-1\}\cup \{\sigma_i: i=r+1, \cdots, r_1-1\}\cup \{\sigma_i: r_1+1, \cdots, r+\ell-1\},
\]
and 
\[
 \a_1=: \a_{\Id}^{J_1, \emptyset}=\{\Delta_1, \cdots, \Delta_{r_1}, \Delta_{r_1+1}^{+}, \cdots, \Delta_{r+\l}^+\},
\]
where $\a=\{\Delta_1, \cdots, \Delta_{r+\l}\}$ with $\Delta_1\unlhd \Delta_2\unlhd \cdots \unlhd \Delta_r$ (cf. Definition \ref{def: 7.2.4}). 
\end{definition}

\begin{lemma}
Let $\d=\a+\l_1[k+1]$, then  
\begin{itemize}
 \item we have $\a=\a_1^{(k+1)}$;
 \item and $\mathfrak{X}_{\d}=\coprod_{w\in S_{r+\l}^{J_1, \emptyset}}O_{\a_{w}}$, where 
 $\a_w=\a_{w}^{J_1, \emptyset}\in S(\a_1)$ is the element associated to $w$ by Lemma \cite[Lemma 3.27]{Deng24b}.
\end{itemize} 
\end{lemma}

\begin{proof}
Note that by definition we have  
 \[
  \a=\a_1^{(k+1)}.
 \]
By definition
$\mathfrak{X}_{\d}$ consists of the orbits $O_{\c}$ with $\c\in S(\d)$ 
such that $\varphi_{e(\c)}(k)+\l_1=\varphi_{e(\a)}(k)$, and 
the latter condition implies that there exists $w\in S_{r+\l}^{J_1, \emptyset}$
such that $\c=\a_{w}^{J_1, \emptyset}$. 
\end{proof}

\begin{prop}\label{prop: 7.6.4}
Let  $\d=\a+\l_1[k+1]$ and $W\subseteq V_{\varphi_{\d}, k+1}$ such 
that $\dim(W)=\l_1$ (which implies that $W=V_{\varphi_{\d}, k+1}$). 
Then the composition of morphisms
\[
 \mathfrak{X}_{\d}=(\mathfrak{X}_{\d})_W \xrightarrow{p} E_{\a}''\xrightarrow{\beta''} E_{\varphi_{\a}},
\]
sends $O_{\a_w}\cap (\mathfrak{X}_{\d})_W$ to $O_{\a_w^{(k+1)}}$. 

\end{prop}

\begin{proof}
This is by definition. 
\end{proof}

\begin{prop}
Let $\mu\in \mathcal{P}(\l_1, r_1)$ and 
 $x_{\mu}=\varsigma_2(\mu)=(x_1^{\mu}, \cdots, x_{r_1}^{\mu})$, 
 $y_{\mu}=(y_1^{\mu}, \cdots, y_{\l_1}^{\mu})$ such that 
\[
 \{1, \cdots, r_1+\l_1\}=\{x_1^{\mu}, \cdots, x_{r_1}^{\mu}, y_1^{\mu}, \cdots, y_{\l_1}^{\mu}\}.
\]
 Then 
\[
 (\a_{\mu^{\flat}})^{\sharp}=\sum_{j=1}^{r_0}[b(\Delta_{x_j^{\mu}}), k]+
 \sum_{j=r_0+1}^{r_1}[b(\Delta_{x_j^{\mu}}), k-1]+\sum_{j=1}^{\l_1}[b(\Delta_{y_j^{\mu}}), k+1],
\]
 for definition of $(\a_{\mu^{\flat}})^{\sharp}$, cf. Lemma \ref{lem: 7.6.18}.
\end{prop}

\begin{proof}
Note that by Proposition \ref{prop: 7.6.1}, 
\[
 \a_{\mu^{\flat}}=\sum_{j=1}^{r_0}[b(\Delta_{x_j^{\mu}}), k]+\sum_{j=r_0+1}^{r_1}[b(\Delta_{x_j^{\mu}}), k-1]+\sum_{j=1}^{\l_1}[b(\Delta_{y_j^{\mu}}), k]
\] 
 and 
 \[
  \a_{\mu}=(\a_{\mu^{\flat}})_{\Gamma}
 \]
for $\Gamma=\sum_{j=1}^{r_0}[b(\Delta_{x_j^{\mu}}), k]$. 
Now by construction in Lemma \ref{lem: 7.6.18}, 
\[
 (\a_{\mu^{\flat}})^{\sharp}=\sum_{j=1}^{r_0}[b(\Delta_{x_j^{\mu}}), k]+
 \sum_{j=r_0+1}^{r_1}[b(\Delta_{x_j^{\mu}}), k-1]+\sum_{j=1}^{\l_1}[b(\Delta_{y_j^{\mu}}), k+1].
\] 
 \end{proof}

\begin{prop}\label{prop-combinatorial-formula-BZ}
We have  
 \[
  n(\a_{\mu}, \a_{\mu^{\flat}})=\sharp \{\c\in S(\a_1): \c^{(k+1)}=\a_{\mu^{\flat}}, \c\geq (\a_{\mu^{\flat}})^{\sharp}\}.
 \] 
\end{prop}

\begin{proof}
Consider the composed morphism 
\[
 h: \mathfrak{X}_{\d}=(\mathfrak{X}_{\d})_W \xrightarrow{p} E_{\a}''\xrightarrow{\beta''} E_{\varphi_{\a}},
\]
then the orbits contained in $h^{-1}(O_{\a_{\mu^{\flat}}})$ is indexed by the set  
\[
 \{\c\in S(\a_1): \c^{(k+1)}=\a_{\mu^{\flat}}, \c\geq (\a_{\mu^{\flat}})^{\sharp}\}
\]
Note that 
by  Theorem \ref{cor: 7.4.19} and  Theorem \ref{prop: 7.3.8}, 
the number  
 \[
  n(\a_{\mu}, \a_{\mu^{\flat}})=\sum_i \dim \mathcal{H}^{2i}(\beta''_*(IC(\line{E''_{\a}((\a_{\mu^{\flat}})^{\sharp}))}))_{x}
 \]
for some $x\in O_{\a_{\mu^{\flat}}}$. Finally, note that the morphism 
$\beta''$ is smooth when restricted to the variety $\beta''^{-1}(O_{\a_{\mu^{\flat}}})$.
Moreover, the fibers are open in some Schubert variety, therefore, we are reduced to the counting of orbits.
\end{proof}

\section{Parabolic Case}\label{sec-Para-case}

In this section,  we deduce a formula for calculating the coefficient $n(\b, \a)$ in the Parabolic cases. 
Most results in this section are stated without proof since the proof are just  direct generalizations of the 
Grassmanian case.

Let 
\[
J\subseteq S
\] 
be a subset of generators and 
$$
\a=\a_{\Id}^{J, \emptyset}
$$
be some multisegment of parabolic type $(J, \emptyset)$ associated to the identity  satisfying that $\ell_{\a, k}\neq 0, \ell_{\a,k+1}=0$ (cf. \cite[Notation 5.12]{Deng23}). Recall that 
\[
\ell_{\a, k}=\sharp\{\Delta\in \a: e(\a)=k\}.
\]
In particular, we have 
\[
\ell_{\a, k}=\ell_{e(\a), k}.
\]


\begin{definition}\label{def-parabolic-multi}
Let $\a(k)=\{\Delta_1, \cdots, \Delta_{\l_k}\}$ with $\Delta_1\unlhd \cdots \unlhd \Delta_{\l_k}$ and $r_0\in \N$ with $1\leq r_0\leq \l_k$. Then let 
\begin{align*}
 \a_1&=(\a\setminus \a(k))\cup \{\Delta\in \a(k):  \Delta\unlhd \Delta_{\l_k-r_0}\}\cup \{\Delta^+\in \a(k): \Delta\unrhd \Delta_{\l_k-r_0+1}\},\\
 \a_2&=(\a\setminus \a(k))\cup \{\Delta^-\in \a(k):  \Delta\unlhd \Delta_{r_0}\}\cup \{\Delta\in \a(k): \Delta\unrhd \Delta_{r_0+1}\}
\end{align*}
and $J_i(r_0, k)(i=1, 2)$ be a subset of $S$ such that 
 $\a_i$ is a multisegment of parabolic type $(J_i(r_0, k), \emptyset)$. Moreover, let 
\[
\a_{\Id}^{J_i(r_0, k), \emptyset}=\a_i,   \text{ for } i=1,2.
\]
\end{definition}

\begin{lemma}
Let $\l_1=\l_k-r_0$ and $\d=\a+\l_1[k+1]$, then  
\begin{itemize}
 \item we have $\a=\a_1^{(k+1)}$;
 \item and $\mathfrak{X}_{\d}=\coprod_{w\in S_{n}^{J_1(r_0, k), \emptyset}}O_{\a_{w}}$, where 
 $\a_w=\a_{w}^{J_1(r_0, k), \emptyset}\in S(\a_1)$ is the element associated to $w$ by \cite[Lemma 3.27]{Deng24b}.
\end{itemize} 
\end{lemma}

\begin{prop}\label{prop: 7.7.1}
Let $w\in S_{n}^{J_2(r_0, k), \emptyset}$. Then 
there exists $w^{\flat}\in S_{n}^{J, \emptyset}$, such that 
\[
 \{\b\in S(\a): \a_{w}\preceq_k \b\}=\{\a_{v}: v\in S_{n}^{J, \emptyset},~ v\leq w^{\flat}\}.
\]
More explicitly, if 
$\a_{w}(k-1)=\{\Delta_1, \cdots, \Delta_{\l_{k-1}}\}$ with $\Delta_1\unlhd \cdots \unlhd \Delta_{\l_{k-1}}$, then 
\[
 \a_{w^{\flat}}=(\a_{w}\setminus \a_{w}(k-1))\cup \{\Delta^+\in \a_{w}(k-1):  \Delta\unlhd \Delta_{r_0}\}\cup \{\Delta\in \a_{w}(k-1): \Delta\unrhd \Delta_{r_0+1}\}.
\]
\end{prop}

\begin{prop}
Let $w\in S_{n}^{J_2(r_0, k), \emptyset}$.
 Then 
\[
 (\a_{w^{\flat}})^{\sharp}=(\a_{w^{\flat}}\setminus \a_{w}(k))\cup \{\Delta^+: \Delta\in \a_{w}(k)\}
\]
 for definition of $(\a_{w^{\flat}})^{\sharp}$, cf. Lemma \ref{lem: 7.6.18}.
\end{prop}
\begin{definition}
Let $t_w\in S_{n}^{J_1(\ell _k-r_0, k), \emptyset}$ be the element such that 
\[
\a_{t_w}=(\a_{w^{\flat}})^{\sharp}.
\]
\end{definition}

\begin{prop}\label{prop: 7.7.8}
Let $P_J$ and $P_{J_1(\l_{k}-r_0, k)}$ be the parabolic subgroups 
corresponding to $J$, $J_1(\l_{k}-r_0, k)$ respectively. Consider the 
natural morphism 
\[
 \pi: P_{J_1(\l_{k}-r_0, k)}\backslash GL_{n}\rightarrow P_{J}\backslash GL_{n}.
\]
Then 
\[
 n(\a_{w}, \a_{v})=\sum_i \dim \mathcal{H}^{2i}(\pi_*(IC(\line{P_{J_1(\l_{k}-r_0, k)}t_{w}B})))_x
\]
for some $x\in P_JvB$.

\end{prop}

\begin{proof}
Consider the composed morphism 
\[
 h: \mathfrak{X}_{\d}=(\mathfrak{X}_{\d})_W \xrightarrow{p} E_{\a}''\xrightarrow{\beta''} E_{\varphi_{\a}}.
\]
This proposition can be deduced in a similar fashion as in the symmetric cases \cite[Corollary 4.15]{Deng23}.
 We omit the details.
\end{proof}

\section{Calculation of Partial BZ operator}\label{Explict-calculation-irred}
In this section, we restrict ourselves to the case of multisegment of parabolic type.

\begin{definition}
Let $J_1\subseteq J_2\subseteq S$ be two subsets of generators of $S_n$. Let 
$v\in S^{J_1, \emptyset}_n, w\in S_n^{J_2, \emptyset}$, we define $\theta_{J_2}^{J_1}(w, v)$ to be the 
multiplicities of $IC(\line{P_{J_2}wB})$ in $\pi_*(IC(\line{P_{J_1}vB}))$, where   
\[
\pi: P_{J1}\backslash GL_n \rightarrow P_{J_2}\backslash GL_n
\]
be the canonical projection. 
\end{definition}

\remk
By \cite[Proposition 2.20]{Deng24b}, in case when $J_1=\emptyset, J_2=\{s_i\}$ we have 
$\theta_{J_2}^{J_1}(w, v)=\mu(s_iw, v)$ if $\ell(v)\leq \ell(s_iv)$, where $\mu(x, y)$ is the 
coefficient of degree $(\ell(y)- \ell(x)-1)/2$ in $P_{x, y}(q)$.

%
%
%

\begin{prop}
Let $J\subseteq S$ be a subset of generators in $S_n$. 
Let $k\in \Z$ and $\a$ be a multisegment satisfies all the assumptions we made before Definition \ref{def-parabolic-multi}.
Then for any $w\in S_n^{J, \emptyset}$, we have 
\[
\D^{k}(L_{\Phi(w)})=\sum_{r_0=0}^{\ell_{\a, k}}\sum_{v\in S_n^{J_2(r_0, k), \emptyset}}\theta_{J}^{J_{1}(\l_{\a, k}-r_0, k)}(w, t_v)L_{\Phi(v)}.
\]
\end{prop}

\begin{proof}
Note that by Proposition \ref{prop: 7.4.2},
\[
\D^k(\pi(\Phi(w)))=\sum_{\b\preceq_k \Phi(w)}n(\b, \a)L_{\b}.
\]
Note that by Proposition  \ref{prop: 7.3.5}, 
$\b\preceq_k \Phi(w)$ implies that 
\[
\b=\Phi(v), 
\]
for some $v\in J_2(\l_{\a, k}-r_0, k)$. 
Moreover, according to Proposition  \ref{prop: 7.7.8}
\[
n(\Phi(v), \Phi(w))=\sum_i \dim \mathcal{H}^{2i}(\pi_*(IC(\line{P_{J_1(\l_{\a, k}-r_0, k)}t_{v}B})))_x
\]
for some $x\in P_JwB$.
In fact, by the decomposition theorem, we have 
\addtocounter{theo}{1}
\begin{equation}\label{eq: 7.42}
\pi_*(IC(\line{P_{J_1(\l_{\a, k}-r_0, k)}t_{v}B})=\bigoplus_{u\in S_n^J} \oplus_{i}IC(\line{P_JuB})^{h_i(u, t_v)}[d_u^i]
\end{equation}
therefore 
\[
\theta_{J}^{J_1(\l_{\a, k}-r_0, k)}(u, t_v)=\sum_{i}h_i(u, t_v).
\]
Furthermore, we denote 
\[
\theta_{J}^{J_1(\l_{\a, k}-r_0, k)}(u, t_v)(q)=\sum_{i}h_i(u, t_v)q^{-d_u^i/2}.
\]
By localizing at a point of $P_JwB$ and applying proper base change, we get 
\addtocounter{theo}{1}
\begin{equation}\label{eq: 7.6}
\sum_{S_{J_1(\l_{\a, k}-r_0, k)}\backslash S_{J}\ni \rho}q^{\ell(\rho)}P^{J_1(\l_{\a, k}-r_0, k), \emptyset}_{\rho w, t_v}(q)=\sum_{u} \theta_{J}^{J_1(\l_{\a, k}-r_0, k)}(u, t_v)(q) P^{J, \emptyset}_{w, u}(q).
\end{equation}
Here $\ell(\rho)$ denotes the length function on symmetric group.
Now we return to the formula 
\addtocounter{theo}{1}
\begin{equation}\label{eq: 7.5}
\pi(\Phi(w))=\sum_{u} P^{J, \emptyset}_{w, u}(1)L_{\Phi(u)}.
\end{equation}
By induction, we can assume that for $u>w$, we have that $L_{\Phi(v)}$ appears in $\D^k(L_{\Phi(u)})$ with multiplicity $\theta_{J}^{J_1(\l_{\a, k}-r_0, k)}(u, t_v)$. Then 
by applying the derivation $\D^k$ to  (\ref{eq: 7.5}), 
on the right hand side we get the multiplicity of $L_{\Phi(v)}$ given by 
\[
x+\sum_{u>w}  \theta_{J}^{J_1(\l_{\a, k}-r_0, k)}(u, t_v) P^{J, \emptyset}_{w, u}(1),
\]
where $x$ denotes the multiplicity of $L_{\Phi(v)}$ in the operator $\D^k(L_{\Phi(w)})$.  On the right hand side, 
applying \cite[Corollary 5.38]{Deng23},
we get
\[
\sum_{S_{J_1(\l_{\a, k}-r_0, k)}\backslash S_{J}\ni \rho}P^{J_1(\l_{\a, k}-r_0, k), \emptyset}_{\rho w, t_v}(1).
\]
Now compare with  (\ref{eq: 7.6}) to get $x=\theta_{J}^{J_1(\l_{\a, k}-r_0, k)}(w, t_v)$.
 \end{proof}

From now on we consider the operator $\D^k(L_{\c})$ for a general multisegment $\c$ such that 
$\ell_{\c, k}>0$.   

\begin{prop}\label{prop-red-to-para}
There exists a multisegment $\c'$ which is of parabolic type $(J_1(\c), \emptyset)$(cf. \cite[Definition 3.26]{Deng24b}).
and  a sequence of integers $k_1, \dots, k_r, k_{r+1}, \dots, k_{r+\l}$ such that $L_{\c}$ is the minimal degree term with multiplicity one in 
\[
{^{k_1}\D}\cdots {^{k_r}\D}\D^{k_{r+1}}\cdots \D^{k_{r+\l}}(L_{\c'}),
\]
and 
\begin{align*}
&\ell_{\c', i}=\ell_{\c, i}, \quad \text{ if } i\leq k,  \\
&\ell_{\c, k+1}=0,\\
&k_i>k+1, \quad \text{if } i>r.
\end{align*}
\end{prop}

\begin{proof}
Let $i_0=\min \{i:  \ell_{b(\c), i}>1\}$ and $\Delta_0=\max_{\prec}\{\Delta\in \c:   b(\Delta)=i_0\}$.
Then replace all segments $\Delta\in \c$ with $b(\c)<i_0$ by ${^{+}\Delta}$ 
and $\Delta_0$ by ${^+\Delta}$ to get a new multisegment $\c_1$. Let $\{i\in b(\c): i<i_0\}=\{j_1<\cdots<j_r\}$, then we have 
$L_{\c}$ is the minimal degree terms in 
\[
{^{j_1-1}\D}\cdots {^{j_r-1}\D}(L_{\c_1}),
\]
Repeat this procedure to get $\c_0$ together with a sequence of integers $k_1, \cdots, k_r$ such that 
$L_{\c}$ is the minimal degree term with multiplicity one in 
\[
{^{k_1}\D}\cdots {^{k_r}\D}(L_{\c_0}).
\]
Suppose that $\ell_{\c_0, k+1}>0$. Now replace all segments $\Delta$ in 
$\c_0$ with $e(\Delta)>k$ by $\Delta^+$ to obtain $\c'$, we are done.
\end{proof}

\begin{definition}
We define
\[
\Gamma^i(\a, k)=\{\b\in \Gamma(\a, k): \deg(\b)+i=\deg(\a)\},
\]
where $\Gamma(\a, k)$ is defined in Definition \ref{def-gamma-a-k}.
\end{definition}

\begin{definition}
Let $\a$ be a multisegment and $k, k_1\in \Z$. Then we define 
\begin{align*}
\Gamma^i(\a, k)_{k_1}&=\{\b\in \Gamma^i(\a, k):  \b\in S(\b)_{k_1}, \b^{(k_1)}\in \Gamma^i(\a^{(k_1)}, k)\}, \\
\Gamma(\a, k)_{k_1}&=\cup_i \Gamma^i(\a, k)_{k_1}.
\end{align*}
More generally for a sequence of integers $k_1, \cdots, k_r$, we define 
\[
\Gamma(\a, k)_{k_1, \cdots, k_r}=\{\b\preceq_k \a:  \b^{(k_1, \cdots, k_{i-1})}\in \Gamma(\a^{(k_1, \cdots, k_{i-1})}, k)_{k_{i}} \text{ for } 1\leq i\leq r\}.
\]
Similarly, we can define 
${_{k_1}\Gamma(\a, k)}$ and 
${_{k_1, \cdots, k_r}\Gamma(\a, k)}$.
\end{definition}

\remk 
We can also define $_{k_{r+1}, \cdots, k_{r+\l}}(\Gamma(\a, k)_{k_1, \cdots, k_r})$, cf. \cite[Definition 6.6]{Deng23}.

\begin{lemma}\label{lem: 7.8.6}
Let $k_1\neq k-1$, then the map
\begin{align*}
\psi_{k_1}: &\Gamma(\a, k)_{k_1}\rightarrow \Gamma(\a^{(k_1)}, k)\\
& \b\mapsto \b^{(k_1)}
\end{align*}
is bijective.
\end{lemma}

\begin{proof}
In fact we have $\Gamma^i(\a, k)=S(\a_i)$ where $\a_i$ is constructed in the following way:
let $\a(k)=\{\Delta_1\succeq \cdots \succeq \Delta_r\}$, then 
\[
\a_{i}=(\a\setminus \a(k))\cup \{\Delta_j^{-}: j\leq i\}\cup \{\Delta_j: j>i\}.
\]
Note that $\Gamma^i(\a, k)=S(\a_i)$, which implies that we have 
\[
\Gamma^i(\a, k)_{k_1}=S(\a_i)_{k_1}.
\]
Finally, note that  by  \cite[Proposition 5.39]{Deng23}  we have a bijection
\[
\psi_{k_1}: S(\a_{i})_{k_1}\rightarrow S(\a_{i}^{(k_1)}).
\] 
Note that $k_1\neq k-1, k$ implies that $\a_i^{(k_1)}\in \Gamma^i(\a^{(k_1)}, k)$ and 
\[
\Gamma(\a^{(k_1)}, k)=\bigcup_{i}S(\a_i^{(k_1)}).
\]
If $k_1=k$, then 
\[
\Gamma(\a, k)_k=S(\a)_k, \quad \Gamma(\a^{(k_1)}, k)=S(\a^{(k)}).
\]
\end{proof}

\begin{lemma}
Let $k_1, k\in \Z$ then the map
\begin{align*}
{_{k_1}\psi}: &{_{k_1}\Gamma(\a, k)}\rightarrow \Gamma({^{(k_1)}\a}, k)\\
& \b\mapsto {^{(k_1)}\b}
\end{align*}
is bijective.
\end{lemma}

\begin{proof}
If $k_1\neq k$, the proof  is the same as that of the previous lemma. Consider the case where $k_1=k$.
Let $\a(k)=\{\Delta_1\succeq \cdots \succeq \Delta_{r_0}\succ \underbrace{[k]=\cdots =[k]}_{r_1}\}$.
Then for $i\leq \l_k$, we have 
\[
\a_{i}=(\a\setminus \a(k))\cup \{\Delta_j^{-}: j\leq i\}\cup \{\Delta_j: j>i\},
\]
where $\Delta_j=[k]$ if $j> r_0$.
And we have $\Gamma^i(\a, k)=S(\a_i)$. By definition,  we have $\b\in {_k\Gamma^i(\a, k)}$ if and only if 
\[
\b\in {_kS(\b)}, \quad ^{(k)}\b\in \Gamma^{i}(^{(k)}\a, k).
\]
Since ${^{(k)}\a}(k)=\{\Delta_1, \cdots, \Delta_{r_0}\}$, we know that for $\b\in {_k\Gamma^i(\a, k)}$, we must have 
$i\leq r_0$. Let
\[
(^{(k)}\a)_{i}=(^{(k)}\a\setminus {^{(k)}\a}(k))\cup \{\Delta_j^{-}: j\leq i\}\cup \{\Delta_j: r_0\geq j>i\}.
\]
Then we have $\Gamma^{i}(^{(k)}\a, k)=S((^{(k)}\a)_{i})$ and
\[
^{(k)}\a_i=(^{(k)}\a)_i.
\]
Finally, we conclude that $\b\in {_k\Gamma^i(\a, k)}$ if and only if $\b\in {_kS(\a_i)}$.
Since the map 
\[
{_kS(\a_i)}\rightarrow S(^{(k)}\a_i)
\]
is bijective, we are done.
\end{proof}

\begin{prop}\label{prop-explicite-determination-partial-der}
Let $\b, \c$ be two multisegments and $k_1\in \Z$ such that 
\[
\b=^{(k_1)}\c, \quad \c\in {{_{k_1}}S(\c)}.
\]
If we write
\addtocounter{theo}{1}
\begin{equation}\label{eq: 7.8.8}
\D^k(L_{\c})=L_{\c}+\sum_{\d\in \Gamma(\c, k)\setminus \{\c\}} \tilde{n}(\d, \c)L_{\d},
\end{equation}
then 
\[
\D^k(L_{\b})=L_{\b}+\sum_{\d\in {_{k_1}\Gamma(\c, k)}\setminus \{\c\}}\tilde{n}(\d, \c)L_{{^{(k_1)}\d}}.
\]
\end{prop}
\begin{proof}
-Suppose that $\deg(\c)=\deg(\b)+1$.
In fact, by Corollary \ref{cor: 3.5.2}, we have 
\[
^{k_1}\D(L_{\c})=L_{\c}+L_{\b}
\]
By applying the derivation $\D^k$ and using the fact $\D^k({^{k_1}\D})={^{k_1}\D}\D^k$, we have 
\[
\D^{k}(L_{\c})+\D^{k}(L_{\b})=L_{\c}+L_{\b}+\sum_{\d\in \Gamma(\c, k)\setminus \{\c\}}\tilde{n}(\d, \c){^{k_1}\D(L_{\d})}
\]
By assumption that $\deg(\b)+1=\deg(\c)$, we have 
\[
^{k_1}\D(L_{\d})=L_{\d}+L_{^{(k_1)}\d} \text{ or } L_{\d},
\]
where $^{k_1}\D(L_{\d})=L_{\d}+L_{^{(k_1)}\d}$ if and only if $\d\in {_{k_1}S(\d)}$ and $\deg(^{(k_1)}\d)=\deg(\d)-1$.
This is equivalent to say that $\d\in {_{k_1}\Gamma(\a, k)}$.

\bigskip
-For general case,  consider 
\[
\{\Delta\in \c:  b(\Delta)=k_1\}=\{\Delta_1\succeq \dots\succeq \Delta_r\}.
\]
Now by Proposition \ref{teo: 3.0.6} and Proposition \ref{prop: 7.4.2}, 
\[
^{k_1}\D(L_{\c})=L_{\b}+\sum_{ \ell_{\d, k_1}>\ell_{\b, k_1}}\tilde{n}(\d, \c)L_{\d},
\]
for some $\tilde{n}(\d, \c)\in \N$. 

If $k_1\neq k$, then 
We observe that for any $\d$ such that $\ell_{\d, k_1}>\ell_{\b, k_1}$ and $\d'\preceq_k \d$, we have 
\[
\ell_{\d', k_1}>\ell_{\b, k_1}, 
\] 
which implies that $L_{\d'}$ can not be a summand of $\D^{k}(L_{\b})$.  Therefore
\[
\D^k(L_{\b})
\]   
is the sum of all irreducible representations $L_{\d''}$ contained in $\D^k(^{k_1}\D)(L_{\c})$ satisfying 
\[
\ell_{\d'', k_1}=\ell_{\b, k_1}.
\]  
Applying the derivation $^{k_1}\D$ to  (\ref{eq: 7.8.8}), we get 
\[
{^{k_1}\D} \D^k(L_{\c})={^{k_1}\D}(L_{\c})+\sum_{\d\preceq_k \c}\tilde{n}(\d, \c) ({^{k_1}\D})(L_{\d}).
\]
Note that in this case the sub-quotient of ${^{k_1}\D} \D^k(L_{\c})$ consisting of irreducible representations $L_{\d''}$ 
satisfying 
\[
\ell_{\d'', k_1}=\ell_{\b, k_1}
\] 
is given by 
\[
L_{\b}+\sum_{\d\in {_{k_1}\Gamma(\c, k)}\setminus \{\c\}}\tilde{n}(\d, \c)L_{{^{(k_1)}\d}}.
\]
Compare the equation ${^{k_1}\D} \D^k(L_{\c})=\D^k(^{k_1}\D)(L_{\c})$ gives the results.\\
If $k_1=k$,  consider 
\[
\{\Delta\in \c:  b(\Delta)=k_1\}=\{\Delta_1\succeq \dots\succeq \Delta_r\}.
\]
Let $\c'$ be the multisegment obtained by replacing all segments $\Delta$ in $\c$
such that $b(\Delta)<k_1$ by ${^{+}\Delta}$, and $\Delta_1$ by ${^{+}\Delta_1}$.
Then there exists 
\[
k_2=k_1-1>k_3>\cdots >k_r
\]
such that 
\[
\c={^{(k_r, \cdots, k_2)}\c'},
\]
and 
\[
\b={^{(k_r, \cdots, k_3, k_1, k_2, k_1)}\c'}.
\]
Let $\b'=^{(k_1)}\c'$, then by induction on $f_{b(\c)}(k)$, we can assume that 
\[
\D^k(L_{\b'})=L_{\b'}+\sum_{\d\in {_{k}\Gamma(\c', k)}\setminus \c'} \tilde{n}(\d, \c')L_{^{(k)}\d}.
\]
Applying what we have proved before, we get 
\[
\D^k(L_{\b})=L_{\b}+\sum_{\d\in _{k_r, \cdots, k_3, k, k_2, k}\Gamma(\c', k)\setminus \{\c'\}}\tilde{n}(\d, \c')L_{^{(k_r, \cdots, k_3, k, k_2, k)}\d}.
\]
Also, we have 
\[
\D^k(L_{\c})=L_{\c}+\sum_{\d\in _{k_r, \cdots, k_3, k_2}\Gamma(\c', k)\setminus \{\c'\}} \tilde{n}(\d, \c')L_{^{(k_r, \cdots, k_3, k_2)}\d}.
\]
Since for any multisegment $\d$, we have
\[
^{(k, k_r, \cdots, k_3, k_2)}\d=^{(k_r, \cdots, k_3, k, k_2, k)}\d,
\]
it remains to show that 
\[
_{k_r, \cdots, k_3, k, k_2, k}\Gamma(\c', k)=_{k, k_r, \cdots, k_3, k_2}\Gamma(\c', k).
\]
By definition and the following lemma, 
we can assume that $r=2$. In this case we argue by contradiction.
Suppose that $\d\in {_{k, k-1, k}\Gamma^i(\c', k)}$
and $\d\notin {_{k, k-1}\Gamma(\c', k)}$, which
is equivalent to say that $\d\notin {_{k,k-1}S(\d)}$.
 Note that $\d\notin {_{k, k-1}S(\d)}$ implies that 
there exists two linked segments $\{\Delta, \Delta'\}$, such that 
\[
b(\Delta)=k, \quad b(\Delta')=k-1.
\]
Then ${^{(k-1, k)}\d}$ contains the pair of segments $\{{^{-}\Delta}, {^{-}\Delta'}\}$. The fact that 
${^{(k-1, k)}\d}\in {_{k}S(^{(k-1, k)}\d)}$ implies  that  ${^{-}\Delta'}=\emptyset$, i.e. $\Delta'=[k-1]$. 
However, this implies that ${^{(k, k-1, k)}\d}\notin \Gamma^i(^{(k, k-1, k)}\c', k)$
since $\deg({^{(k, k-1, k)}\d})+i=\deg({^{(k, k-1, k)}\a})+1$,
which is a  contradiction.

Conversely, assume that $\d\in {_{k, k-1}\Gamma(\c', k)}$ and $\d \notin {_{k, k-1, k}\Gamma^i(\c', k)}$, 
which by definition is equivalent to $\d\notin {_{k, k-1, k}S(\d)}$. Note that $\d\notin {_{k, k-1, k}S(\d)}$ implies 
that $\d\notin {_{k}S(\d)}$, which contradicts to $\d\in {_{k, k-1}S(\d)}$.  
\end{proof}

\begin{lemma}
Let $k>k-1>k'$ be two integers. Then for any multisegment $\c$, we have 
\[
{_{k, k'}\Gamma(\c, k)}={_{k', k}\Gamma(\c, k)}.
\]

\end{lemma}
\begin{proof}
Note that since for any multisegment $\d$
\[
{^{(k', k)}\d}={{^{k, k'}}\d}, 
\]
the fact 
\[
{_{k, k'}\Gamma(\c, k)}={_{k', k}\Gamma(\c, k)}
\]
is equivalent to 
\[
\d\in {_{k, k'}S(\d)}\Leftrightarrow \d\in {_{k', k}S(\d)}
\]
for all $\d\in {_{k, k'}\Gamma(\c, k)}$. But for any multisegment $\d$ and $k>k-1>k'$, we have 
\[
\d\in {_{k, k'}S(\d)}\Leftrightarrow \d\in {_{k', k}S(\d)}.
\]
Hence we are done.
\end{proof}

\begin{prop}\label{prop: 7.8.10}
Let $k_1\neq k-1, k, k+1$.
Let $\b, \c$ be two multisegments such that 
\[
\b=\c^{(k_1)}, \quad \c\in S(\c)_{k_1}.
\]
If we write
\addtocounter{theo}{1}
\begin{equation}
\D^k(L_{\c})=L_{\c}+\sum_{\d\in \Gamma(\c, k)\setminus \{\c\}} \tilde{n}(\d, \c)L_{\d},
\end{equation}
then 
\[
\D^k(L_{\b})=L_{\b}+\sum_{\d\in \Gamma(\c, k)_{k_1}\setminus \{\c\}}\tilde{n}(\d, \c)L_{d^{(k_1)}}.
\]
\end{prop}
\begin{proof}
The proof is the same as  Proposition \ref{prop-explicite-determination-partial-der}.
\end{proof}

Now let $\c'=\Phi(w)$ for some $w\in S_n^{J, \emptyset}$.

\begin{cor} \label{coro-fornula-derivative}
We have 
\begin{align*}
\D^k(L_{\a})=\sum_{r_0=0}^{\l_{\a, k}}&\sum_{v\in S_n^{J_2(r_0, k), \emptyset}, \Phi(v)\in _{k_1, \cdots, k_r}(\Gamma(\Phi(w), k)_{k_{r+1}, \cdots, k_{r+\l}})}\\&\theta_{J}^{J_{1}(\l_k-r_0, k)}(w, t_v)L_{^{(k_1, \cdots, k_r)}\Phi(v)^{(k_{r+1}, \cdots, k_{r+\l})}}.
\end{align*}
\end{cor}

\begin{notation}\label{nota7-8-12}
For $\b\preceq_k \a$, we denote 
\[
\theta_k(\b, \a)=\theta_{J}^{J_{1}(\l_k-r_0, k)}(w, t_v)
\]
if $\b=^{(k_1, \cdots, k_r)}\Phi(v)^{(k_{r+1}, \cdots, k_{r+\l})}$. Otherwise, put $\theta_k(\b, \a)=0$.

\end{notation}
\remk 
The same way we define ${_{k}\theta}(\b, \a)$ by the formula
\[
(^{k}\D)(L_{\a})=\sum_{\b}{_{k}\theta}(\b, \a)L_{\b}.
\]
And let 
\[
\Gamma(k, \a)=\{\b: {_{k}\theta}(\b, \c)\neq 0 \text{ for some }\c\in S(\a)\},
\]
it shares similar properties with $\Gamma(\a, k)$.

\appendix

\section{Minimal Degree Terms in Partial BZ operator}

\begin{prop}\label{teo: 3.0.6}
\begin{description}
\item[(i)] 
Suppose that $\a$ satisfies the hypothesis $H_{k}(\a)$ (cf. \cite[Definition 5.3]{Deng23}).\\
Then $\D^{k}(L_{\a})$ contains  in $\mathcal{R}$ a unique 
irreducible representation of minimal degree, which is
 $L_{\a^{(k)}}$,
 and it appears with multiplicity one.
\item[(ii)]If $\a$ fails to satisfy the hypothesis $H_{k}(\a)$, then \\
 $L_{\a^{(k)}}$
 will not appear in $\D^{k}(L_{\a})$, 
 and the irreducible representations appearing
 are all of degree $>\deg(\a^{(k)})$. 
\end{description}
\end{prop}

\begin{proof}
Let $\a=\{\Delta_1\preceq \cdots \preceq \Delta_r\}$, such that 
\[
 e(\Delta_1)\leq \cdots<e(\Delta_i)=\cdots=e(\Delta_j)<\cdots \leq e(\Delta_r),
\]
with $k=e(\Delta_i)$.

We prove the proposition by induction on $\ell(\a)$(cf. Definition \ref{def: 1.2.10}).
For, $\ell(\a)=0$, which means that $\a=\a_{\min}$, 
in this case $\a$ satisfies the $H_{k}(\a)$, and 
\[
 \D^{k}(L_{\a})=\D^{k}(\pi(\a))=
 \Delta_{1}\times \cdots \times (\Delta_{i}+\Delta_{i}^{-})\times \cdots 
 \times (\Delta_{j}+\Delta_{j}^{-}) \times \cdots  
\]
which contains
\[
 L_{\a^{(k)}}=\pi(\a^{(k)})=\Delta_{1}\times \cdots \times \Delta_{i}^{-}
 \times \cdots \Delta_{j}^{-}\times \cdots .
\]
Hence we are done in this case.

Next we consider general $\a$.
We write
\addtocounter{theo}{1}
\begin{equation}\label{eq: 1}
\pi(\a)=L_{\a}+ \sum_{\b<\a}m(\b,\a)L_{\b}. 
\end{equation}
Now applying $\D^{k}$ to both sides and consider only the lowest degree
terms,
on the left hand side, we get
\addtocounter{theo}{1}
\begin{equation}\label{eq: 2}
\pi(\a^{(k)})=\Delta_{1}\times \cdots \times \Delta_{i-1}\times \Delta_{i}^{-}\times \cdots \times \Delta_{j}^{-}\times \cdots \Delta_{r}.
\end{equation}
By \cite[Theorem 2.22]{Deng23}, both sides are a nonnegative sum 
of irreducible representations, then 

\begin{itemize}
\item 
If  $\a$ satisfies the hypothesis $H_{k}(\a)$, 
on the right hand side by combining \cite[Lemma 5.5]{Deng23} and induction, we know that for all $\b<\a$,
$\D^{k}(L_{\b})$ does not contain
$L_{\a^{(k)}}$ as subquotient. Hence
$\D^{k}(L_{\a})$
must contain $L_{\a^{(k)}}$ with multiplicity one. 
We have to show that it does not contain 
other subquotients of $\pi(\a^{(k)})$. 
Note that by induction, we have the following formula
\[
 \pi(\a^{(k)})=X+\sum_{\c\in S(\a)_{k}\setminus \a}m(\c, \a)L_{\c^{(k)}},
\]
where $X$ denotes the minimal degree terms in $\D^{k}(L_{\a})$. 
Now apply \cite[Corollary 5.40]{Deng23},  we conclude that $X=L_{\a^{(k)}}$.

\item 
Now if $\a$ fails to satisfy the hypothesis $H_{k}(\a)$, $\a\notin S(\a)_k$,
combining \cite[Proposition 5.39]{Deng23} and induction, we know that there exists $\b\in S(\a)_k$, such that
$\a^{(k)}=\b^{(k)}$ and $\D^{k}(L_{\b})$ contains $L_{\a^{(k)}}$ as a subquotient with
multiplicity one.

Now by the Lemma \ref{lem: 3.0.8}, $\pi(\a)-\pi(\b)$ is a nonnegative sum 
of irreducible representations which contain $L_{\a}$: by the positivity of partial BZ operator, 
 we obtain a nonnegative sum of irreducible representations after
applying $\D^{k}$. Now
$$
\D^{k}(\pi(\a)-\pi(\b))=\pi(\a^{(k)})-\pi(\b^{(k)})+\text{ higher degree terms}
$$ 
contains only terms of degree
$>\deg(\a^{(k)})$, so does $\D^{k}(L_{\a})$.

\end{itemize}
This finishes our arguments.

\end{proof}

\begin{cor}\label{cor: 3.5.2}
Let $\a$ be a multisegment such that $\varphi_{e(\a)}(k)=1$. Then 
\begin{itemize}
 \item If $\a\in S(\a)_k$, then $\D^k(L_{\a})=L_{\a}+L_{\a^{(k)}}$. 
 
 \item If $\a\notin S(\a)_k$, then $\D^k(L_{\a})=L_{\a}$.
\end{itemize}
 
\end{cor}

\begin{proof}
First of all, we observe 
that the highest degree term in $\D^k(L_{\a})$ is given by 
$L_{\a}$. In fact, we have 
\[
 \D^k(\pi(\a))=\D^k(L_{\a})+\sum_{\b<\a}m(\b, \a)\D^k(L_{\b}),
\]
meanwhile we have 
\[
 \D^k(\pi(\a))=\pi(\a)+\text{ lower terms. }
\]
By induction on $\ell(\a)$ we conclude that the highest degree terms
in $\D^k(L_{\a})$ is $L_{\a}$.

If $\a\in S(\a)_k$, then Proposition \ref{teo: 3.0.6} implies that 
the minimal degree term of $\D^k(L_{\a})$, but since 
$\deg(\a^{(k)})=\deg(\a)-1$, therefore we must have 
\[
 \D^k(L_{\a})=L_{\a}+L_{\a^{(k)}}.
\]
On the contrary, if  $\a\notin S(\a)_k$, then 
by (ii) of the Proposition \ref{teo: 3.0.6}, we know that 
all irreducible representations appearing in 
 $\D^k(L_{\a})$ are of degree $>\deg(\a^{(k)})=\deg(\a)-1$, which implies 
 \[
 \D^k(L_{\a})=L_{\a}.
\]

\end{proof}

\bibliographystyle{plain}
\bibliography{biblio}

\end{document}